\documentclass[a4paper,11pt,reqno]{amsart}
 \usepackage{amssymb,amsmath,amstext,amsthm,amsfonts}
\usepackage{euscript,graphicx,a4wide}

\usepackage{hyperref}  

\theoremstyle{plain}
\newtheorem{maintheorem}{Theorem}

\newtheorem{theorem}{Theorem}[section]
\newtheorem{proposition}[theorem]{Proposition}
\newtheorem{corollary}[theorem]{Corollary}
\newtheorem{lemma}[theorem]{Lemma}

\newtheorem{remark}[theorem]{Remark}

\newtheorem{example}{Example}
\theoremstyle{definition}

\newtheorem{conjecture}{Conjecture}

\newcommand{\bR}{{\bf R}}

\newcommand{\RR}{{\mathbb R}}

\newcommand{\CC}{{\mathbb C}}
\newcommand{\ZZ}{{\mathbb Z}}

\newcommand{\sS}{{\mathbb S}}

\newcommand{\fX}{{\EuScript{X}}}

\newcommand{\cM}{{\mathcal M}}
\newcommand{\cU}{{\mathcal U}}
\newcommand{\cO}{{\mathcal O}}

\newcommand{\cH}{{\mathcal H}}

\newcommand{\ov}{\overline}
\newcommand{\de}{{\delta}}
\newcommand{\vfi}{{\varphi}}
\renewcommand{\epsilon}{\varepsilon}

\newcommand{\qand}{\quad\text{and}\quad}
\newcommand{\wh}{\widehat}

\DeclareMathOperator{\diam}{diam}
\DeclareMathOperator{\spec}{sp}
\DeclareMathOperator{\dist}{dist}
\DeclareMathOperator{\supp}{supp}
\DeclareMathOperator{\per}{Per}
\DeclareMathOperator{\sing}{Sing}
\DeclareMathOperator{\crit}{Crit}
\DeclareMathOperator{\lip}{Lip}

\theoremstyle{remark}

\title[All Lyapunov exponents with constant sign]
{Sinks and sources for $C^1$ dynamics whose Lyapunov
  exponents have constant sign}

\thanks{
  The author was partially supported by
  CNPq-Brazil (grant 301392/2015-3) and FAPESB-Bahia-Brazil
  (grant PIE0034/2016).}

\date{\today}

\author[Vitor Araujo]{Vitor Araujo}
\email{vitor.araujo.im.ufba@gmail.com}
\urladdr{https://sites.google.com/site/vdaraujo99/}
\address{Instituto de Matem\'atica e Estat\'{\i}stica,
  Universidade Federal da Bahia, Av. Ademar de Barros s/n,
  40170-110 Salvador, Brazil.}



\begin{document}

\begin{abstract}
  Let $f:M\to M$ be a $C^1$ map of a compact manifold $M$, with
  dimension at least $2$, admitting some point whose future trajectory
  has only negative Lyapunov exponents. Then this trajectory converges
  to a periodic sink. We need only assume that $Df$ is never the null
  map at any point (in particular, we need no extra smoothness
  assumption on $Df$ nor the existence of a invariant probability
  measure), encompassing a wide class of possible critical
  behavior. Similarly, a trajectory having only positive Lyapunov
  exponents for a $C^1$ diffeomorphism is itself a periodic repeller
  (source).
  
  Analogously for a $C^1$ open and dense subset of vector field on
  finite dimensional manifolds: for a flow $\phi_t$ generated by such
  a vector field, if a trajectory admits weak asymptotic sectional
  contraction (the extreme rates of expansion of the Linear Poincar\'e
  Flow are all negative), then this trajectory belongs either to the
  basin of attraction of a periodic hyperbolic attracting orbit (a
  periodic sink or an attracting equilibrium); or the trajectory
  accumulates a codimension one saddle singularity. Similar results
  hold for weak sectional expanding trajectories.

  Both results extend part of the non-uniform hyperbolic theory
  (Pesin's Theory) from the $C^{1+}$ diffeomorphism setting to $C^1$
  endomorphisms and $C^1$ flows. Some ergodic theoretical consequences
  are discussed. The proofs use versions of Pliss' Lemma for maps and
  flows translated as (reverse) hyperbolic times, and a result
  ensuring that certain subadditive cocycles over $C^1$ vector fields
  are in fact additive.
\end{abstract}

\maketitle

\tableofcontents

\section{Introduction and statements of the results}
\label{sec:introd-statem-result}

In what follows $M$ is a connected compact finite
$d$-dimensional manifold $M$, with $d\ge2$, endowed with a
Riemannian metric $\langle\cdot,\cdot\rangle$ which induces
a norm $\|\cdot\|$ on the tangent bundle of $M$ and a
distance $\dist$ on $M$, and a volume form $m$ that we call
Lebesgue measure. For any subset $A$ of $M$ we denote by
$\ov{A}$ the (topological) closure of $A$.

We extend the following well-known result from Nonuniform
Hyperbolic (Pesin's) Theory (see e.g.~\cite[Corollary
S.5.2]{KH95} or~\cite[Corollary 15.4.2]{BarPes2007}) to
$C^1$ endomorphisms and $C^1$ singular vector fields.
\begin{theorem}
  Let $f$ be a H\"older-$C^1$ diffeomorphism of $M$ and
  $\mu$ a $f$-invariant ergodic probability measure such
  that all its Lyapunov exponents are negative
  (respectively, positive). Then $\supp\mu$ is an attracting
  (respectively, repelling) periodic orbit.
\end{theorem}
Results along this line for one-dimensional transformations
usually assume at least the same amount of extra smoothness:
see e.g. Ma\~n\'e~\cite{Man85};
Campanino~\cite{campanino1980} and
Przytycki~\cite{Przyty93}. Other results assume only $C^1$
smoothness but have dimensional restrictions; see
e.g.~\cite{barramorales}.

In all these results the existence of a invariant
  probabilty measure is another standing assumption which we
  mostly avoid in the main statements, but explore some of
  its consequences in what follows.

  As a consequence of these results, points with negative
  asymptotic rates of expansion belong to the basin of a
  periodic attracting orbit, which is its stable manifold.
  In contrast, we note that in~\cite{BoCroSh13}
  the authors show that, for generic
  $C^1$-diffeomorphisms, hyperbolic measures having positive
  and negative Lyapunov exponents do not necessarily admit
  (un)stable invariant manifolds.

\subsection{The discrete time case}
\label{sec:easy-local-diffeom}

For $C^1$ maps on compact manifolds we obtain a necessary
and sufficient condition for a given trajectory to be on the
basin of an attracting periodic orbit
from asymptotic information on the derivative.

Let $f:M\to M$ be a $C^1$ map such that
$\inf_{x\in M}\|Df(x)\|>0$.  The Subadditive Ergodic Theorem
(see e.g. \cite{mane1983,ViOl16}) ensures that the largest
asymptotic growth rate
\begin{align*}
\chi(x)=\lim_{n\to+\infty}\ln\|Df^n(x)\|^{1/n}
\end{align*}
exists for all $x$ on a total probability subset since
$\ln^+\|Df\|=\max\{0,\ln\|Df\|\}$ is $\mu$-integrable for
each $f$-invariant probability measure $\mu$.

In what follows we write
$ A^-_k(x) =
\liminf_{n\to+\infty}n^{-1}\sum_{j=0}^{n-1}\ln\|Df^k(f^{kj}x)\|.
$

We recall that $p\in M$ belongs to a periodic orbit (with
period $\tau$) if there exists $\tau\in\ZZ^+$ so that
$f^\tau p=p$. This periodic orbit
$\cO_f(p)=\{p,fp, \dots, f^{\tau-1}p\}$ is attracting (a
sink, for short) if there exists a neighborhood $V_p$ of $p$
such that $f^\tau\mid_{V_p}: V_p\to V_p$ is a contraction:
there exists $0<\lambda<1$ so that $\dist(f^\tau q,f^\tau
r)<\lambda\dist(q,r), \forall q,r\in V_p$. Equivalently,
$\|Df^{\tau}(p)\|<\lambda$ for some $\lambda\in(0,1)$.

The basin of attraction of a sink $\cO_f(p)$ is the
following subset
$B(\cO_f(p))=\{x\in M: \omega(x)=\cO_f(p)\}$, where the
omega-limit $\omega(x)$ of $x$ is the set of accumulation
points of the positive orbit of $x$:
$y\in\omega(x)\iff\exists n_k\nearrow\infty:
f^{n_k}x\xrightarrow[k\to\infty]{} y$.

\begin{maintheorem}\label{mthm:negativexp-sink}
  Let $f:M\to M$ be a $C^1$ map such that
  $\inf_{x\in M}\|Df(x)\|>0$.  Then $x\in M$ is contained in
  the basin of attraction of a attracting periodic orbit (a
  sink) if, and only if, $A^-_k(x)<0$ for some $k\in\ZZ^+$.
\end{maintheorem}



Coupling the pointwise result above with the Subadditive
Ergodic Theorem and ergodic decomposition, we deduce:
  \begin{corollary}
    \label{cor:negativexp-sink}
    Let $\mu$ be an invariant probability measure with
    respect to a $C^1$ map $f:M\to M$ such that
    $\inf_{x\in M}\|Df(x)\|>0$ and $\chi(x)<0,
    \mu$-a.e. $x\in M$.  Then $\mu$ decomposes as
    $\tilde\mu+\sum_{i\ge1}\mu_i$, where each $\mu_i$ is a
    Dirac mass equidistributed on a periodic attracting
    orbit of $f$ (a sink), the sum is over at most countably
    many such orbits, and $\tilde\mu$ (which might be the
    null measure) satisfies $A^-_1(x)\ge0,
    \tilde\mu$-a.e. $x\in M$.  In addition, if $\mu$ is
    $f$-ergodic, then $\mu$ is concentrated on the orbit of
    a periodic attractor (sink).
  \end{corollary}
  
  \begin{remark}\label{rmk:notHolderC1}
    \begin{enumerate}
    \item \emph{We do not need H\"older continuity of the
        derivative} in the arguments proving
      Theorem~\ref{mthm:negativexp-sink} and
      Corollary~\ref{cor:negativexp-sink}.
    \item We do not need that $f$ be a diffeomorphism or
      local diffeomorphism; compare with Corollary S.5.2 of
      \cite[Supplement]{KH95} where the usual H\"older
      condition on the derivative of a diffeomorphism in
      \emph{Pesin's Theory}, or non-uniform hyperbolic
      theory, is used to construct \emph{hyperbolic blocks}.
    \item We need only to assume that $Df(x)$ is not the
      null map for all $x\in M$, and this weak condition is
      compatible with a wide class of critical points of a
      smooth map $f\in C^1(M,M)$.
    \item The previous assumption ensures that
      $A_k^-(x)>k\ln\inf_{x\in M}\|Df(x)\|>-\infty$ for all
      $x\in M$ and all $k\ge1$, and so also
      $\chi(x)\ge\ln\inf_{x\in M}\|Df(x)\|>-\infty$ on a
      total probability subset of points $x$.
    \end{enumerate}
  \end{remark}

  \subsubsection{The diffeomorphism case}
\label{sec:diffeomorphism-case}

If $f$ is a $C^1$ diffeomorphism, then exchanging $f$ with
$f^{-1}$ we have that
$ \tilde\chi(x) = \lim_{n\to+\infty}\ln\|Df^{-n}(x)\|^{1/n}
= \lim_{n\to+\infty}\ln\|Df^{n}(x)^{-1}\|^{1/n}$ exists for
$x$ on a total probability subset of $M$ and gives the least
asymptotic growth rate. We also write
$\tilde A^-_k(x)= \liminf_{n\to\infty}n^{-1}
\sum_{j=0}^{n-1}\ln\|\big(Df^k(f^{kj}x)\big)^{-1}\|$. We say
that a periodic orbit of $f$ is \emph{repelling} if it is an
attracting periodic orbit for $f^{-1}$.

\begin{maintheorem}\label{mthm:negexpdiffeo}
  Let $f:M\to M$ be a $C^1$ diffeomorphism. Then $x\in M$
  belongs to a repelling periodic orbit (a source) if, and
  only if, $\tilde A^-_k(x)<0$ for some $k\in\ZZ^+$.
\end{maintheorem}
We easily deduce the following ergodic consequence
from the above pointwise result.
  \begin{corollary}
    \label{cor:decomposes}
    Let $\mu$ be an invariant probability measure with
    respect to a $C^1$ diffeomorphism $f:M\to M$. Then it
    admits a decomposition\footnote{Some of the summands in
      the decomposition might be null.}
    $\mu=\tilde\mu+\sum_{i\ge1}\nu_i+\sum_{j\ge1}\rho_i$,
    where each $\nu_i$ (respectively, $\rho_i$) is a Dirac
    mass equidistributed on a periodic attracting
    (resp. repelling) orbit of $f$, both sums are over at
    most countably many such orbits, and $\tilde\mu$
    satisfies $\chi(x)\ge 0\ge-\tilde\chi(x)$ for
    $\tilde\mu$-a.e. $x\in M$.

    In particular, if $\mu$ is non-atomic, then
    $\mu=\tilde\mu$ and so either $\mu$ has some zero
    exponent, or $\mu$ is a hyperbolic measure with
    exponents of different signs.
  \end{corollary}

\subsubsection{Robustness of negative Lyapunov exponents}
\label{sec:robustn-negative-lya}
  
In contrast to the results above, it is well-known that
positive Lyapunov exponents in all directions on a total
probability set for a $C^1$ local diffeomorphism of $M$
imply that $f$ is a uniformly expanding map. More precisely,
see \cite{alves-araujo-saussol,cao2003}, if
$\tilde\chi(x)=\lim_{n\to+\infty}\ln\|Df^n(x)^{-1}\|^{1/n}<0$
for $\mu-$ a.e. $x\in M$ with respect to every $f$-invariant
probability measure $\mu$, then we can find constants
$C,\sigma>0$ so that $\|Df^n(x)^{-1}\|\le C e^{-\sigma n}$
for all $x\in M, n\ge1$. This is a robust situation: the
assumptions automatically hold for a $C^1$-neighborhood of
such local diffeomorphisms; see e.g. \cite{ViOl16}.

In the same setting exchanging positive with negative
exponents in all directions we obtain the following.

\begin{maintheorem}
  \label{mthm:allnegexpfinitesinks}
  If a $C^1$ map $f:M\to M$ is such that
  $\inf_{x\in M}\|Df(x)\|>0$ and for all $f$-invariant
  probability measures $\mu$ we have $\chi(x)<0,
  \mu$-a.e. $x\in\ M$,
  then there exists a unique periodic attracting orbit
  $\cO_f(p)$ whose basin is $M$. 
\end{maintheorem}

See next Subsection~\ref{sec:conjectures} for comments and
corollaries of this.
  
\subsection{The case of (singular) vector fields}
\label{sec:case-singular-vector}

Let $\fX^1(M)$ be the space of $C^1$ vector fields on $M$
which are inwardly transverse to the boundary endowed with
the $C^1$ topology and $\phi_t$ be flow generated by
$G\in\fX^1(M)$.  We denote by $D\phi_t$ \emph{the derivative
  of $\phi_t$ with respect to the ambient variable $q$ and
  set $D_q \phi_t=D\phi_t(q)$}.  An analogous Subadditive
Ergodic Theorem also holds:
$\chi_G(x)=\lim_{T\to+\infty}\ln\|(D\phi_T(x)\|^{1/T}$ exists
on a total probability subset.

To state analogous results for vector fields we need some
preliminary notions about critical elements of the flow
induced by a vector field and, since the vector field
direction always has zero Lyapunov exponent for every
invariant probability measure, we need to deal with the
derivative cocycle of the flow $\phi_t$ generated by the
vector field $G$ restricted to the normal direction to the
flow: these notions can be defined for a flow on any finite
dimensional Riemannian manifold.

\subsubsection{Some preliminary notions}
\label{sec:some-prelim-notions}

Given $G \in {\fX}^1(M)$, where $M$ is a compact finite
dimensional Riemannian manifold with dimension $d\ge2$, we
denote by \emph{$DG$ the derivative of the vector field $G$}
with respect to the ambient variable $q$, and when
convenient we write $D_q G$ for the derivative $DG$ at $q$,
also denoted by $DG_q$, where $DG_q v=\nabla_v G(y)$ where
$\nabla$ is the unique Levi-Civita connection compatible
with the Riemannian metric on $M$.  Given $q \in M$ an orbit
segment $\{\phi_tq; a \leq t \leq b\}$ is denoted by
$\phi_{[a,b]}q$.

\subsubsection*{Critical elements}

An \emph{equilibrium} or \emph{singularity} for $G$ is a
point $\sigma\in M$ such that $\phi_y(\sigma)=\sigma$ for
all $t\in\RR$, i.e.  a fixed point of all the flow maps,
which corresponds to a zero of the associated vector field
$G$: $G(\sigma)=\vec0$. We denote by
$\sing(G)=\{x\in M:G(x)=\vec0\}$ the set of singularities of
$G$.  Every point $p\in M$ which is not a singularity, that
is $p$ satisfies $G(p)\neq0$, is a \emph{regular} point for
$G$.

An \emph{orbit} of $G$ is a set
$\cO(q)=\cO_G(q)=\{\phi_tq: t\in\RR\}$ for some $q\in
M$. Hence $\sigma\in M$ is a singularity of $G$ if, and only
if, $\cO_G(\sigma)=\{\sigma\}$.  A \emph{periodic orbit} of
$G$ is an orbit $\cO=\cO_G(p)$ such that $\phi_Tp=p$ for some
minimal $T>0$ (equivalently $\cO_G(p)$ is compact and
$\cO_G(p)\neq\{p\}$). We denote by $\per(G)$ the set of all
periodic orbits of $G$.

A \emph{critical element} of a given
vector field $G$ is either a singularity or a periodic
orbit.  The set $\crit(G)=\sing(G)\cup\per(G)$ is the set of
\emph{critical elements} of $G$.

\subsubsection*{Limit sets. Attractors}

If $q\in M$, we define \emph{omega-limit} set $\omega_G(q)$
as the set of accumulation points of the positive orbit
$\{\phi_tq:t\geq 0\}$ of $q$.  We also define the
alpha-limit set $\alpha_G(q)=\omega_{-G}$, where $-G$ is
the time reversed vector field $G$, corresponding to the set
of accumulation points of the negative orbit of $q$.

A subset $\Lambda$ of $M$ is \emph{invariant} for $G$ (or
$G$-invariant) if $\phi_t\Lambda=\Lambda, \forall
t\in\RR$. We note that $\omega_G(q), \alpha_G(q), \sing(G)$
and their complements in $M$ are $G$-invariant.

For every compact invariant set $\Lambda$ of $X$ we define
the \emph{stable set} of $\Lambda$
$$
W^s_G(\Lambda)=\{q\in M:\omega_G(q)\subset \Lambda\},
$$
and also its \emph{unstable set}
$$
W^u_G(\Lambda)=\{q\in M:\alpha_G(q)\subset \Lambda\}.
$$
A compact invariant subset $\Lambda$ of $G$ is
\emph{attracting} if $\Lambda_G(U)=\cap_{t\geq 0}\phi_t(U)$
equals $\Lambda$ for some neighborhood $U$ of $\Lambda$
satisfying $\overline{\phi_t(U)}\subset U, \forall t>0$. In
this case the neighborhood $U$ is called an \emph{isolating
  neighborhood} of $\Lambda$. Analogously, $\Lambda$ is
\emph{repelling} if it is attracting for $-G$. We say
$\Lambda$ is a proper subset if
$\emptyset\neq\Lambda\neq M$.

\subsubsection*{Hyperbolic critical elements}

A (hyperbolic) \emph{sink} of $G$ is a singularity which is
also an attracting set, it is a trivial attracting set of
$G$.  A \emph{source} of $G$ is a trivial repelling subset
of $G$, i.e. a singularity which is attracting for $-G$.

A \emph{singularity $\sigma$ is hyperbolic} if the
eigenvalues of $DG(\sigma)$, the derivative of the vector
field at $\sigma$, have real part different from zero.  In
particular, sinks and sources are hyperbolic singularities,
since all the eigenvalues of the former have negative real
part and those of the latter have positive real part.

A \emph{periodic orbit $\cO_G(p)$ of $G$ is hyperbolic} if
the eigenvalues of $D\phi_T(p):T_pM\to T_p M$ (the
derivative of the diffeomorphism $\phi_T$ at $p$ with $T>0$
the period of $p$) are all different from $1$.

When a critical element is hyperbolic, then its stable and
unstable sets have the structure of an immersed manifold (a
consequence of the Stable Manifold Theorem, see
e.g. \cite{PM82}), and are known as \emph{stable}
and \emph{unstable manifolds}.

In the particular case of attracting critical elements, the
corresponding stable set (manifold) is also known as its
(topological) \emph{basin}.

\subsubsection*{Linear Poincar\'e Flow.}

If $x$ is a regular point of a $C^1$ vector field
$G$ (i.e. $G(x)\neq \vec0$), denote by
$
  N_x=\{v\in T_xM: \langle v , G(x)\rangle=0\}
$
the orthogonal complement of $G(x)$ in $T_xM$.  Denote by
$O_x:T_xM\to N_x$ the orthogonal projection of $T_x M$ onto
$N_x$.  For every $t\in\RR$ define, see Figure~\ref{fig:LPF}
\begin{align*}
P_x^t:N_x\to N_{\phi_t x}
\quad\text{by}\quad
P_x^t=O_{\phi_t x}\circ D\phi_t(x).
\end{align*}
\begin{figure}[h]
  \centering\label{fig:LPF}
  \includegraphics[width=12cm]{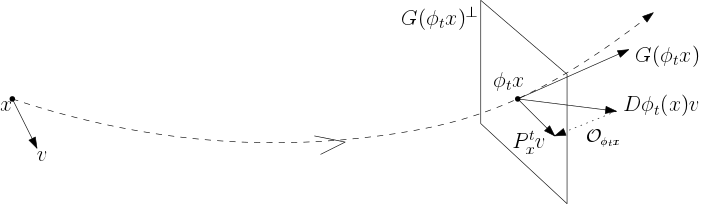}
  \caption{Sketch of the Linear Poincar\'e flow $P^t_x$ of a
  vector $v\in T_xM$ with $x\in M\setminus\sing(G)$.}
\end{figure}
It is easy to see that $P=\{P_x^t:t\in\RR, G(x)\neq 0\}$
satisfies the cocycle relation
$ P^{s+t}_x=P^t_{\phi_s x}\circ P^s_x $ for every
$t,s\in\RR$.  The family $P=P_G$ is called the {\em Linear
  Poincar\'e Flow} of $G$.

\begin{remark}
  \label{rmk:semiflow}
  The Linear Poincar\'e Flow does not immediately extends to
  the smooth semiflow setting and this is an important tool
  in our proofs; see Conjecture~\ref{conj:continuous}.
\end{remark}

\subsubsection{Negative (positive)  exponents and sinks}
\label{sec:negative-exponents-s}

First we consider setting similar to
Theorems~\ref{mthm:negativexp-sink}
and~\ref{mthm:negexpdiffeo}.  In what follows we write
$\chi_G^-(x) = \liminf_{T\to+\infty}\ln\|D\phi_T(x)\|^{1/T}$
and
$\tilde\chi_G^-(x) = \liminf_{T\to+\infty}\ln\|D\phi_T(x)^{-1}\|^{1/T}$.

\begin{maintheorem}
  \label{mthm:negexpflow-sink}
  Given $G\in\fX^1(M)$ suppose that $x\in M$ satisfies
  $\chi_G^-(x)<0$. Then there exists a hyperbolic sink
  $\sigma\in\sing(G)$ so that
  $\phi_tx\to\sigma$ as $t\to\infty$.
  Otherwise, suppose that $\tilde\chi_G^-(x)<0$. Then $x$ is a 
  repelling equilibrium (a source).
\end{maintheorem}

Since the flow $\phi_t$ induced by a vector field
$G\in\fX^1(M)$ satisfies
$D\phi_t(x)G(x)=G(\phi_tx), x\in M, t\in\RR$, it is natural
to consider trajectories which have asymptotic contraction
along all transversal directions to the vector field, which we
refer to as \emph{sectional asymptotic contraction}.

\subsubsection{Negative sectional exponents and sinks}
\label{sec:negative-section-exp}

However, weak sectional asymptotic contraction along a given
trajectory does not necessarily implies that this trajectory
converges to a sink, either a singularity or a periodic
orbit, as the following example shows.

\begin{example}\label{ex:Bowen}
  Consider the vector field known as ``Bowen example''; see
  e.g. \cite{Ta95} and Figure~\ref{fig:bowen}.
  \begin{figure}[htpb]
    \centering
    \includegraphics[width=11cm]{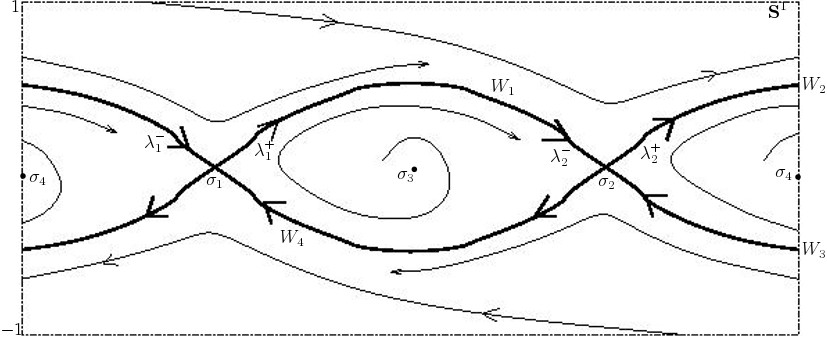}
    \caption{A sketch of Bowen's example flow.}
    \label{fig:bowen}
  \end{figure}
  This vector field is inwardly transverse to the boundary
  of $M=\sS^1\times [-1,1]$. Let
  $W=\cup_{i=1}^4W_i$ be the set formed by the heteroclinic
  connections between and including the equilibria
  $\sigma_1,\sigma_2$.  The future trajectories under the
  corresponding flow $\phi_t$ of every $z\in M\setminus W$
  accumulates on either side of the heteroclinic
  connections, as suggested in the figure, if we impose the
  condition
  $\lambda_1^{-} \lambda_2^{-} > \lambda_1^{+}
  \lambda_2^{+}$ on the eigenvalues of the saddle equilibria
  $\sigma_1$ and $\sigma_2$ (for more specifics on this see
  \cite{Ta95} and references therein) so that $\sigma_i$ are
  area contracting: $|\det D\phi_t(\sigma_i)|\to0$
  exponentially fast with $t>0$, $i=1,2$.

  It is well-known (see \cite{Ta95} for more details) that
  the time taken by the orbit $\phi_tx$ of any point $x$ in
  the connected components of $\sS^2\setminus W$ containing
  one of $\sigma_3,\sigma_4$, with exception of the
  equilibria $\sigma_3,\sigma_4$, while passing through a
  small neighborhood of either $\sigma_1$ or $\sigma_2$ is
  much larger than all the previous history of the
  orbit. Then the rate $\ln\|P_x^T\|^{1/T}$ oscillates
  between the value of $\lambda_i^+$ (when approaching) and
  $\lambda_i^-$ (at departure) at each passage near $\sigma_i$,
  $i=1,2$, that is
  $ \liminf_{T\to\infty}\ln\|P^T_x\|^{1/T}<
  0<\limsup_{T\to\infty}\ln\|P^T_x\|^{1/T}.  $

  All points $z$ on the connected components of
  $\sS^2\setminus W$ containing the boundary of $M$ also
  accumulate $W$ and thus satisfy the same asymptotic rates.
\end{example}

The previous Example~\ref{ex:Bowen} motivates us to state a
partial analogue to Theorem~\ref{mthm:negativexp-sink} in
the vector field setting.

\begin{maintheorem}
  \label{mthm:weaknegexpFlow}
  Let $G\in\fX^1(M)$ be such that $\sing(G)$ (possibly
  empty) is hyperbolic.  If $x\in M\setminus\sing(G)$
  satisfies $\liminf_{T\to\infty}\ln\|P_x^T\|^{1/T} <0$,
  then
  \begin{itemize}
  \item either $x$ is contained in the basin of attraction of a
    sink: either an attracting equilibrium or a hyperbolic
    periodic attracting orbit; 
  \item or the orbit of $x$ accumulates a hyperbolic
    codimension $1$ saddle singularity\footnote{The stable
    manifold of the singularity has codimension one as an
    immersed submanifold of $M$.}.
  \end{itemize}
\end{maintheorem}

To obtain the same conclusion as
Theorem~\ref{mthm:negativexp-sink} for a sectional
contracting trajectory of a vector field, we need to assume
a stronger condition on the asymptotic contracting rate.

\begin{maintheorem}
    \label{mthm:strongnegexpFlow}
    Let $G\in\fX^1(M)$ be such that $\sing(G)$ is
    hyperbolic.  If $x\in M\setminus\sing(G)$ is such that
    $\limsup_{T\to\infty}\ln\|P_x^T\|^{1/T} <0$, then $x$
    is contained in the basin of attraction of a sink:
    either an attracting equilibrium or a hyperbolic
    periodic attracting orbit.
\end{maintheorem}

\begin{remark}
  \label{rmk:opendense}
  \begin{enumerate}
  \item We do not need H\"older continuity of the derivative
    in the arguments proving Theorems~\ref{mthm:weaknegexpFlow}
    and~\ref{mthm:strongnegexpFlow} and corollaries.
  \item The condition ``$\sing(G)$ is hyperbolic'' imposed
    on $G$ in the statement of
    Theorems~\ref{mthm:weaknegexpFlow}
    and~\ref{mthm:strongnegexpFlow} is satisfied by an open
    and dense subset of $\fX^1(M)$; see e.g. \cite{PM82}.
  \item In the particular case $\sing(G)=\emptyset$, 
    Theorems~\ref{mthm:weaknegexpFlow}
    and~\ref{mthm:strongnegexpFlow} become the direct
    analogue to Theorem~\ref{mthm:negativexp-sink} in the
    vector field setting: the trajectory of $x$ converges to
    either an attracting fixed point of the flow or to an
    attracting periodic orbit, even if
    asymptotic contraction only holds sectionally.
  \end{enumerate}
\end{remark}

\subsubsection{Positive sectional exponents and sources}
\label{sec:positive-section-exp}

Akin to expanding maps and expanding
measures, for expanding semiflows the asymptotic expansion
condition on a given trajectory does not necessarily implies
that the trajectory is a (periodic) source.

\begin{example}
  \label{ex:Lorenz-semiflow}
  The geometrical Lorenz expanding semiflow introduced by
  Williams~\cite{Wil79} exhibits asymptotic expansion in the
  transversal direction of all positive time trajectories
  not falling into the singularity, has a dense regular
  trajectory and a dense subset of periodic expanding
  trajectories; see~\cite{Wil79} for details.
\end{example}

The analogous to Theorem~\ref{mthm:negexpdiffeo} is also
true for sectional expansion.

\begin{maintheorem}
  \label{mthm:posectionalexp}
  Let $G\in\fX^1(M)$ be such that $\sing(G)$ is hyperbolic.
  If $x\in M\setminus\sing(G)$ satisfies
  $\liminf_{T\to\infty}\ln\|(P_x^T)^{-1}\|^{1/T} <0$, then
  \begin{enumerate}
  \item either $x$ belongs to a hyperbolic periodic
    repelling orbit;
  \item or the orbit of $x$ accumulates a hyperbolic saddle
    singularity of index $1$\footnote{The stable manifold
    of the singularity has dimension one as an immersed
    submanifold of $M$.}.
  \end{enumerate}
  If $x\in M\setminus\sing(G)$ is such that
    $\limsup_{T\to\infty}\ln\|(P_x^T)^{-1}\|^{1/T} <0$, then $x$
    satisfies item (1).
\end{maintheorem}

\begin{remark}
  \label{rmk:bowenexample} Example~\ref{ex:Bowen} also
  provides an instance of item (2) in the statement of
  Theorem~\ref{mthm:posectionalexp}. This example is easily
  adapted to higher dimensions: just multiply Bowen's vector
  field $G$ by a ``North-South'' vector field in the $n$th
  sphere $\sS^n, n\ge1$ to obtain higher dimensional
  instances of Theorems~\ref{mthm:weaknegexpFlow}
  and~\ref{mthm:posectionalexp}.
\end{remark}


\subsection{Comments, corollaries and conjectures}
\label{sec:conjectures}

We comment and state some corollaries of the results in what
follows, and then some conjectures. The proofs of the
corollaries are given later in the text: see next
Subsection~\ref{sec:organization-text} on the organization
of this text.

\subsubsection{The $C^1$ endomorphism setting}
\label{sec:discrete-setting}

\subsubsection*{Negative Lyapunov exponents everywhere}

\begin{corollary}
  \label{cor:compactnegexpperiodic}
  Let $K\subset M$ be a compact $f$-invariant subset such
  that $\chi^-(x)<0$ for all $x\in K$. Then $K$ is the union
  of a finite family of sinks.
\end{corollary}

The setting of Theorems~\ref{mthm:negativexp-sink}
and~\ref{mthm:allnegexpfinitesinks} is robust: there exists
a $C^1$ neighborhood $\cU$ of $f$ such that each $g\in\cU$
satisfies the same assumptions and conclusions.

\begin{example}\label{ex:uniquesink}
  An example of a $C^1$ endomorphism $f:\sS^2\to\sS^2$
  satisfying the conclusion of
  Theorem~\ref{mthm:allnegexpfinitesinks} can be given as
  follows: consider
  \begin{itemize}
  \item $h:\sS^2\circlearrowleft$ the North-South map on
    $\sS^2=\{(x,y,z)\in\RR^3:x^2+y^2+z^2=1\}$, given by the
    time-$1$ map of the gradient flow $\dot w=\nabla\vfi(w)$
    with $\vfi(x,y,z)=z$, where $w=(x,y,z)\in\sS^2$;
  \item $P$ the stereographic projection from $N(0,0,1)$ to
    $\sS^2$.
  \end{itemize}
  Then let $g:\sS^2\to\CC$ be the $C^1$ map that sends
  $\sS^2$ to the half-sphere $\sS^2\cap\{z<0\}$ together
  with the surface $R$ of revolution generated by the three
  half-circles with diameter $1/3$ drawn in the left hand
  side of Figure~\ref{fig:projecao}. We choose
  $g\mid_{\sS^2\cap\{z<0\}}$ to be the identity and
  $g\mid_{\sS^2\cap\{z\ge0\}}$ as the vertical projection
  from $\sS^2\cap\{z\ge0\}$ to $S$.

\begin{figure}[h]
  \centering \includegraphics[width=3.5in]{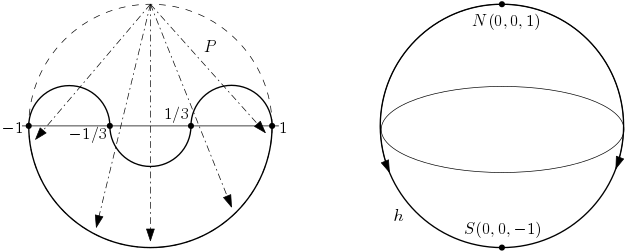}
  \caption{\label{fig:projecao} Maps whose composition
    defines a $C^1$ transformation with negative Lyapunov
    exponents everywhere and a unique sink with full basin.}
\end{figure}

Finally, define $f$ as the composition $h\circ P\circ g$.
Note that the image of $f_0=P\circ g$ is contained in
$\sS^2\cap\{z\le0\}$ and so the image of $f$ is contained in
$\sS^2\cap\{z<0\}$, hence $f$ is a strict contraction.
Moreover, $f_0(N)=S(0,0,-1)$.  Hence $f=h\circ f_0$
contracts distances uniformly and fixes $S$, which is a sink
attracting all points of $\sS^2$.
\end{example}


\subsubsection*{Negative exponents Lebesgue almost
  everywhere}
\label{sec:negative-exponents-l}

It is known that there are $C^1$ open families of local
diffeomorphisms satisfying $\tilde A^-_1(x)<0$ for Lebesgue
almost points of the ambient manifold and which are not
uniformly expanding; see~\cite{Pinheiro05} and references
therein and also \cite[Appendix]{ABV00} for a concrete
example of open classes of such local diffeomorphisms.

\begin{remark}
  \label{rmk:lebNUCversusLebexp}
  \begin{enumerate}
  \item It is well-known that for (expanding maps and)
    expanding measures there exists a dense subset of
    periodic sources in its support; see \cite{Pinheiro05}.
  \item It is known that $\tilde\chi<0$ for
    $\mu$-a.e. implies $\tilde A^-_k(x)<0$ $\mu$-a.e. for
    any given $f$-invariant measure $\mu$ and some
    $k\in\ZZ^+$; see e.g. \cite{ABV00}. It is conjectured
    that $\tilde A^-_1(x)<0$ for $m$-a.e. $x$ implies the
    existence of an $f$-invariant probability measure $\mu$
    satisfying $\tilde\chi(x)<0, \mu$-a.e. $x$; see
    e.g~\cite{Pinheiro05} and more
    recently~\cite{pinheiro18}.
  \end{enumerate}
\end{remark}

In our setting, it is natural to consider $C^1$ maps
satisfying $A^-_k<0$ for Lebesgue almost all points and
some $k\ge1$.

\begin{corollary}
  \label{cor:Leballneg}
  Let $f:M\to M$ be a $C^1$ map such that
  $\inf_{x\in M}\|Df(x)\|>0$. Then $A^-_k(x)<0,
  m$-a.e. $x\in M$ for some $k\in\ZZ^+$ if, and only if,
  there exists an at most countable family of Dirac masses
  concentrated on periodic attracting orbits (sinks) whose
  basins form an open, dense and also a full Lebesgue
  measure subset of $M$.
\end{corollary}

\begin{example}\label{ex:infiniteattractors}
  An example of a diffeomorphism satisfying the conclusion
  of Corollary~\ref{cor:Leballneg} can be constructed by the
  direct product of the map from \cite[Example
  1]{vdaraujo2001} and the North-South map, both on the
  circle.

  The latter is represented in Figure~\ref{fig:N-S} given by
  the time-$1$ map $h$ of the gradient flow
  $\dot z=\nabla \psi(z)$ with $\psi(x,y)=y$ on
  $z=(x,y)\in\sS^1=\{(x,y)\in\RR^2:x^2+y^2=1\}$.  The former can be
  seen as the time-$1$ map $g$ of the gradient flow
  $\dot x=\nabla\vfi(x)$ with
  $\vfi:[-\pi^{-1},\pi^{-1}]\to\RR, t\mapsto t^4\sin(1/t)$,
  (where we identify $\pm\pi^{-1}$ to obtain the circle)
  exhibiting a countable number of attracting fixed sinks
  whose basins cover the entire domain of $\vfi$ with the
  exception of the countably many local maxima $\{m_k\}$ of
  $\vfi$; see Figure~\ref{fig:sinus}.

  \begin{figure}
    \centering 
    \includegraphics[width=3cm]{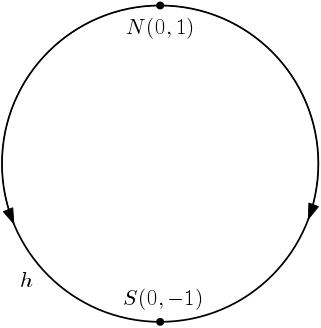}
    \caption{\label{fig:N-S}The North-South map on the circle.}
  \end{figure}

  Then $g\times h:\sS^1\times\sS^1\circlearrowleft$ is a
  $C^1$ map having a countable number of attracting fixed sinks
  whose basins cover the entire space
  $[-\pi^{-1},\pi^{-1}]\times\sS^1$ with the exception of
  the Lebesgue null sets $[-\pi^{-1},\pi^{-1}]\times\{N\}$
  and $\{m_k\}\times\sS^1$.

\begin{figure}[h]
  \centering \includegraphics[width=3.5in]{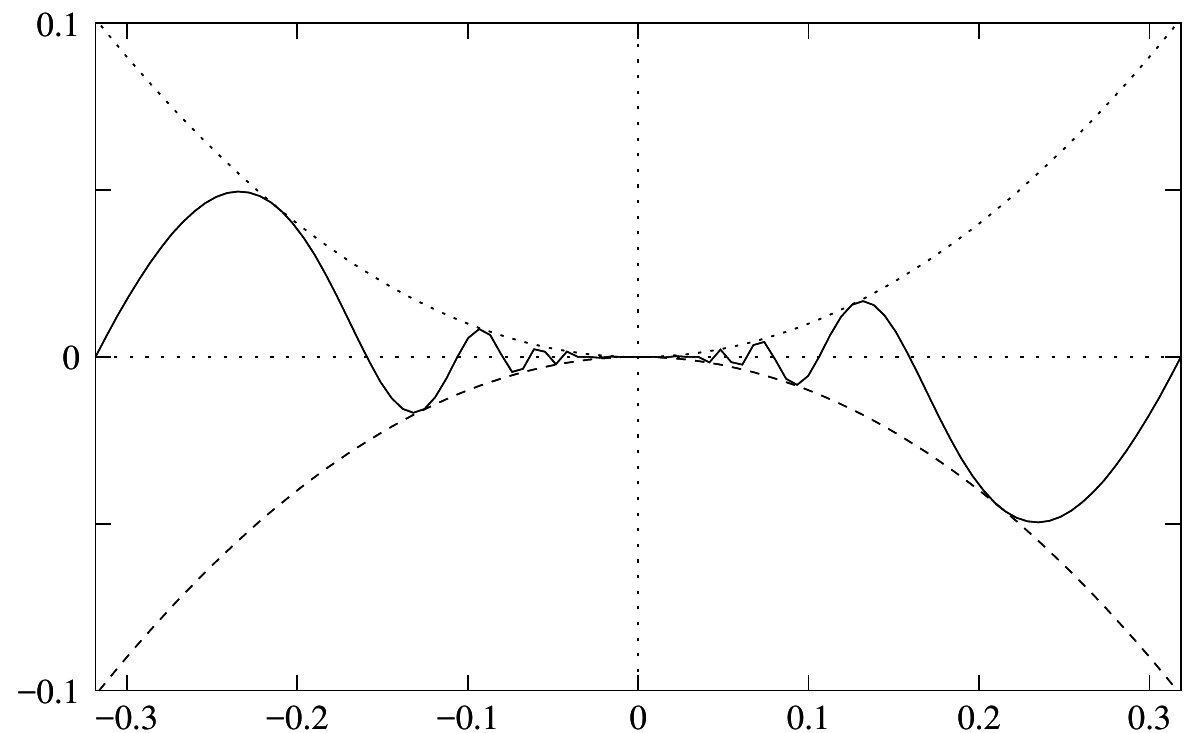}
  \caption{\label{fig:sinus} A map whose gradient flow has
    infinitely many sinks whose basins form a open, dense
    and full measure subset of the ambient space.}
\end{figure}
\end{example}

This situation is however not robust as the following
example shows.

\begin{example}\label{ex:Mora-Viana}
  There exist one parameter families
  $\vfi_\mu:\sS^2\to\sS^2, -1\le\mu\le1$ of smooth
  diffeomorphisms such that (see e.g. \cite[Chapter 5,
  Section 1 \& Chapter 7, Section 2]{PT93} and references
  therein)
  \begin{itemize}
  \item there exists a sink $S\in\sS^2$ whose basin is an
    open, dense and full Lebesgue measure subset of $\sS^2$
    for $-1\le\mu\le0$; and
  \item there are parameters $\mu_n\searrow0$ so that
    $\vfi_{\mu_n}$ admits positive Lebesgue measure subsets
    of points $x$ with a positive Lyapunov exponent (e.g., the
    ergodic basin of H\'enon-like attractors near the
    generic unfolding of a quadratic homoclinic tangency).
  \end{itemize}
\end{example}







\subsubsection{The singular $C^1$ vector field setting}
\label{sec:vector-field-setting}

The pointwise statements of the continuous version of the results take
advantage of the existence of infinitesimal generators of the cocycles
$\ln\|D\phi_t(x)v\|$ and $\ln\|P_x^tv\|$ to avoid assumptions on time
averages; see Section~\ref{sec:linear-poincare-flow-1}.
The cocycle
relation for the derivative of the flow of $G$ and for the Linear
Poincar\'e Flow $P_G$ implies that the functions
\begin{align*}
  \Gamma(v)=\Gamma_G(v)&:\RR\times M\to\RR, \quad (t,x)\mapsto
                   \Gamma_t(x)v=\ln\|D\phi_t(x)v\|;
  \\
  \tilde\Gamma(v)=\tilde\Gamma_G(v)
                 &:\RR\times M\to\RR, \quad (t,x)\mapsto
                   \tilde\Gamma_t(x)v=\ln\|D\phi_t(x)^{-1}v\|;
  \\
\psi(v)=\psi_P(v)&:\RR\times M\setminus\sing(G)\to\RR,
(t,x)\mapsto\psi_t(x)v=\ln\|P^t_xv\|
                   \qand
  \\
\tilde\psi(v)=\tilde\psi_P(v)&:\RR\times M\setminus\sing(G)\to\RR,
(t,x)\mapsto\tilde\psi_t(x)v=\ln\|(P^t_x)^{-1}v\|
\end{align*}
are \emph{additive}:
$\Gamma_{t+s}(y)v \le\Gamma_s(\phi_ty)v+\Gamma_t(y)v$ and
$\psi_{t+s}(x)v \le \psi_s(\phi_tx)v+\psi_t(x)v$ for
$y\in M, x\in M\setminus\sing(G)$ and $t,s\in\RR$; and
similarly for $\tilde\Gamma$ and $\tilde\psi$.
In Section~\ref{sec:flow-case} a more detailed version of
the following is stated and proved, mainly as a consequence of the
extra smoothness gained along trajectories of the flow
generated by a $C^1$ vector field, since these trajectories
become $C^2$ curves.
\begin{theorem}
  \label{thm:additivity}
  The functions $D_G(v):=\lim_{h\to0} h^{-1}\Gamma_h(y)(v)$,
  $\tilde D_G(v):=\lim_{h\to0} h^{-1}\tilde\Gamma_h(y)v$,
  $ D(v):=\lim_{h\to0} h^{-1}\psi_h(x)v$ and
  $\tilde D(v):=\lim_{h\to0} h^{-1}\tilde\psi_h(x)v$ are continuous
  and uniformly bounded on $T_x^1M$. Moreover
  $ \Gamma_t(y)w = \int_0^tD_G(\Phi_sw)\,ds $ and
  $ \psi_t(x)v = \int_0^tD(\wh{\Phi_s}v)\,ds,$ for $t\in\RR$,
  $y\in M$, $w\in T_y^1M$ and
  $x\in M\setminus\sing(G), v\in T_x^1M\cap G^\perp$; and similarly
  for $\tilde\Gamma$ and $\tilde\psi$.
\end{theorem}
Here $T^1M$ os the unit tangent bundle, $G^\perp$ is the normal
bundle to $G$ on $M\setminus\sing(G)$; $\Phi_t$ is the induced
flow on the unit tangent bundle
$\Phi_tv=\frac{D\phi_t(x)v}{\|D\phi_t(x)v\|}$ and $\wh{\Phi_t}$ the
analogous construction with $P_x^t$ in the place of $D\phi_t(x)$.

Hence, for instance, we can replace
$\lim_{T\to\infty}\ln\|P_x^T\|<0, \mu$-a.e $x$ for an ergodic
$G$-invariant probability measure by the condition
$\mu\big(\sup_{v\in T_x^1M\cap G^\perp} D(v)\big)<0$ and so on; see
the statement of Corollary~\ref{cor:decomposes-flow} in what follows.

\subsubsection*{On sectional Lyapunov exponents}

We can also interpret $\mu(D)<0$ as a condition on the
Lyapunov spectrum of $\mu$. The Oseledets Multiplicative
Ergodic Theorem states that Lyapunov exponents exist for the
cocycle $D\phi_t(x)$ for a total probability subset of
points: for any $G$-invariant probability measure and for
$\mu$-a.e. $x$ there exists
$k=k(x)\in\{1, \ldots, d=\dim(M)\}$, numbers
$\chi_1(x)<\cdots<\chi_k(x)$ and a $D\phi_t$-invariant
decomposition $ T_xM=E_x^1\oplus\cdots\oplus E_x^k$ (i.e.,
$ D\phi_tE^i_x=E^i_{\phi_tx}$) so that
\begin{align*}
  \chi(x,v)=\lim_{t\to\pm\infty} \log\|D\phi_t(x)\cdot v\|^{1/t} =
  \chi_i(x), \quad
  \forall v\in E^i_x\setminus\{\vec0\}, 1\le i\le k(x).
\end{align*}
Moreover, $\chi(x,G(x))=0$ for $\mu$-a.e.  $x\in
M\setminus\sing(G)$. In addition, the angles between any two Oseledets
subspaces decay sub-exponentially fast along orbits of $f$ (see
e.g.~\cite[Theorem 1.3.11 \& Remark 3.1.8]{BarPes2007}):
$ \lim_{t\to\pm\infty} \frac 1n \log\sin \angle \left(\bigoplus_{i\in
    I} E_{\phi_tx}^i, \bigoplus_{j\notin I} E_{\phi_tx}^j\right) = 0 $
for any $I \subset \{1, \ldots, k(x)\}$ and $\mu$-a.e. $x$, where for
any given pair $E,F$ of complementary subspaces (i.e.
$E\oplus F=T_xM$) we set
$\cos\angle(E,F):=\inf\left\lbrace |\langle v,w\rangle| :
  \|v\|=1=\|w\|, v\in E, w\in F\right\rbrace$.  This implies, in
particular, that for any $2$-dimensional subspace $S$ of $T_xM$ the
value of $\lim\ln|\det D\phi_t\mid S|^{1/T}$ equals the sum of the two
largest Lyapunov exponents of all basis of $S$. Since the direction of
the flow has zero Lyapunov exponent, the assumptions on
Theorems~\ref{mthm:weaknegexpFlow}, ~\ref{mthm:strongnegexpFlow}
and~\ref{mthm:posectionalexp} can be restated as:
$\liminf_{T\to\infty}\ln|\det D\phi_t\mid S|<0$ for every
two-dimensional subspace $S$ of $T_xM$; or with $\limsup$ etc. This is
why it is natural to label these conditions on a trajectory of a flow
as \emph{asymptotic sectional growth conditions} or \emph{conditions
  on Lyapunov exponents transverse to the vector field}, since
$\mu\big(\sup_{v\in T_x^1M\cap G^\perp}D\big)<0$ for an ergodic
$G$-invariant probability measure $\mu$ amounts to say that the the
Lyapunov exponents are $\mu$-a.e. equal to
$\chi_1<\dots<\chi_{k-1}<0=\chi_k$ for some $k\le d$.

\subsubsection*{Asymptotic contraction (Lebesgue almost)
  everywhere}
\label{sec:asympt-contract-lebe}

The setting of Theorems~\ref{mthm:negexpflow-sink},
~\ref{mthm:weaknegexpFlow},~\ref{mthm:strongnegexpFlow}
and~\ref{mthm:posectionalexp} is robust: on a $C^1$
neighborhood $\cU$ of $G$ in $\fX^1(M)$ we have the same
assumptions and conclusions.

If we replace $f$ by the flow generated by $G\in\fX^1(M)$
and the assumptions $\chi^-(x)<0$ or $A_k^-(x)<0$ by
$\chi_G^-(x)$ on Corollaries~\ref{cor:compactnegexpperiodic}
and~\ref{cor:Leballneg}, then we get the same conclusions in
the vector field setting.
Moreover, since the Linear Poincar\'e Flow is only defined
for regular points, we also have
\begin{corollary}
  \label{cor:compactnegsecexpperiodic}
  Let $G\in\fX^1(M)$ be such that $\sing(G)$ is hyperbolic.
  \begin{enumerate}
  \item If $\limsup_{T\to\infty}\ln\|P_x^T\|^{1/T} <0$
    for\footnote{Since $\sing(G)$ is hyperbolic, then
      $m(\sing(G))=0$ because (if non-empty) $\sing(G)$ is
      finite.}  $m$-a.e. $x\in M$, then an at most countable
    family of periodic attracting orbits or attracting
    equilibria (i.e., an at most enumerable family of
    sinks) whose basins form an open, dense and also a full
    Lebesgue measure subset of $M$.
  \item 
    Let $K\subset M$ be a compact $G$-invariant subset such
    that $\liminf_{T\to\infty}\ln\|P_x\|^{1/T}<0$ for all
    $x\in K$. Then $K$ is the union of a finite family of
    sinks.
  \end{enumerate}
\end{corollary}

  There are many classes of examples of vector fields having
  an open, dense and full Lebesgue measure subset in the
  basin of attraction of a family of sinks and are
  arbitrarily $C^1$ close to a vector field having a
  positive Lebesgue measure subset of trajectories with some
  asymptotic expansion; see e.g.~\cite[Chapter
  9]{BDV2004}. We outline one of these.

\begin{example}
  \label{ex:singcycle}
  Using singular cycles, Morales~\cite{Mo96} studied the
  unfolding of a geometric Lorenz attractor when the
  singularity contained in this attractor goes through a
  saddle-node bifurcation. It is shown in~\cite{Mo96} that
  there exist one-parameter families $(G_t)_{t\in[-1,1]}$ of
  vector field in a $3$-manifold $M$ which \emph{unfold a
    Lorenz attractor directly into a Plykin attractor}. This
  means that there are $\mu\in(-1,1)$ and $\delta>0$ such
  that
  \begin{itemize}
  \item if $t\in[(\mu-\delta,\mu)$, then $G_t$ has a
    geometric Lorenz attractor.
  \item $G_\mu$ is a saddle-node Lorenz vector field. 
  \item if $t\in(\mu,\mu+\delta)$, then $G_t$ is an Axiom A
    vector field (see e.g.\cite{PT93}).
  \end{itemize}
  The vector fields $G_t$ for $t\in(\mu-\delta,\mu]$ satisfy
  $\chi_{G_t}^-(x)>0$ for a positive Lebesgue measure subset
  of points, namely the basin of attraction of the
  (saddle-node) geometric Lorenz attractor. In contrast,
  $G_t$ for $t\in(\mu,\mu+\delta)$ has finitely many
  hyperbolic attractors whose basins form an open, dense and
  full Lebesgue subset of $M$; see e.g.~\cite{Bo75}.
\end{example}

\subsubsection*{Decomposition of invariant probability
  meeasures for vector fields}
\label{sec:decomp-invari-probab}

We can obtain ergodic statements similar to
Corollaries~\ref{cor:negativexp-sink}
and~\ref{cor:decomposes}.
\begin{corollary}
  \label{cor:decomposes-flow}
  Let $\mu$ be an invariant probability measure with respect
  to a $C^1$ vector field $G\in\fX^1(M)$. Then it admits a
  decomposition\footnote{Again, some of the summands in the
    decomposition might be null.}
  $\mu=\tilde\mu+\sum_{i\ge1}\nu_i+\sum_{j\ge1}\rho_i$,
  where each $\nu_i$ (respectively, $\rho_i$) is a Dirac
  mass equidistributed on a periodic attracting
  (resp. repelling) orbit of $G$, both sums are over at most
  countably many such orbits, and $\tilde\mu$ satisfies
  $\chi_G(x)\ge0$ and
  $\tilde\chi_G(x)\ge0$ for $\tilde\mu$-a.e.  $x\in M$.

  In particular, if $\mu$ is non-atomic, then
  $\mu=\tilde\mu$ and so either $\mu$ has some zero
  exponent, or $\mu$ is a hyperbolic measure with exponents
  of different signs.
\end{corollary}

\subsubsection{Conjectures}
\label{sec:conjectures-1}

The pointwise statements of the continuous version of the
results took advantage of the existence of infinitesimal
generators of the cocycles $\ln\|D\phi_t(x)\|$ and
$\ln\|P_x^t\|$ to avoid assumptions on time averages, as in
the statements of Theorems~\ref{mthm:negativexp-sink} and
~\ref{mthm:negexpdiffeo}.

\begin{conjecture}\label{conj:averages}
  In the discrete setting we can argue as in the vector
  field setting to reduce asymptotic growth conditions to
  asymptotic average growth condition. That is, replacing
  the assumptions $A^-(x)<0$ or $\tilde A^-(x)<0$ by
  $\liminf_{n\to\infty}\ln\|Df^n(x)^{\pm1}\|^{1/n}<0$ in the
  statements of Theorems~\ref{mthm:negativexp-sink}
  and~\ref{mthm:negexpdiffeo}.
\end{conjecture}

A positive answer to this would be an advance to answer the
conjecture mentioned in
Remark~\ref{rmk:lebNUCversusLebexp}(2); see
also~\cite{pinheiro18}.

We expect the assumption that $Df$ is never the null map is
an artifact of our proof and can be bypassed.
\begin{conjecture}
  \label{conj:nondeg}
For maps the result of Theorem~\ref{mthm:negativexp-sink} is
still valid without any extra assumptions on the derivative.
\end{conjecture}

We should not need to use hyperbolicity assumptions on the
vector field $G$.

\begin{conjecture}\label{conj:nohyperbolic}
  Theorems~\ref{mthm:weaknegexpFlow},
  ~\ref{mthm:strongnegexpFlow} and~\ref{mthm:posectionalexp}
  hold for all vector fields $G\in\fX^1(M)$.
\end{conjecture}

From Remark~\ref{rmk:bowenexample} we conjecture that
Example~\ref{ex:Bowen} is paradigmatic.

\begin{conjecture}
  \label{conj:Bowen}
  Let $G$ be a vector field satisfying
  $\liminf_{T\to\infty}\frac1T\ln\|(P_x^T)^{-1}\|<0$ and
  also $\liminf_{T\to\infty}\frac1T\ln\|P_x^T\|<0$ for a
  open, dense and full Lebesgue measure subset of $M$. Then
  $G$ exhibits saddle connections similar to
  Example~\ref{ex:Bowen}.
\end{conjecture}

Due to the simple character of the dynamics of sinks and
sources, we should be able to obtain similar results in the
setting of continuous flows and smooth semiflows, not
necessarily generated by vector fields.

\begin{conjecture}
  \label{conj:continuous}
  There exists an open and dense family of continuous flows
  or smooth semiflows on manifolds where an extended notion
  of sectional asymptotic expansion or contraction along
  trajectories ensures the existence of sources or sinks.
\end{conjecture}

\subsection{Organization of the text}
\label{sec:organization-text}

We prove
Theorems~\ref{mthm:negativexp-sink},~\ref{mthm:negexpdiffeo}
and~\ref{mthm:allnegexpfinitesinks} in
Section~\ref{sec:positive-lyapun-expo}, together with
Corollaries~\ref{cor:negativexp-sink},~\ref{cor:decomposes},~\ref{cor:compactnegexpperiodic}
and~\ref{cor:Leballneg}. We state a version of Pliss'
Lemma~\ref{le:pliss} for flows in
Subsection~\ref{sec:pliss-lemma-flows}. In
Subsection~\ref{sec:linear-poincare-flow-1} we translate the
assumptions of Theorems~\ref{mthm:negexpflow-sink},
~\ref{mthm:weaknegexpFlow} and \ref{mthm:strongnegexpFlow}
in a convenient format. Then we use these results as tools
for the proof of the first part of the statement of
Theorem~\ref{mthm:negexpflow-sink} in
Subsection~\ref{sec:asympt-contract-all} and the proof of
Theorems~\ref{mthm:weaknegexpFlow} and
\ref{mthm:strongnegexpFlow} in the remaining
Subsections~\ref{sec:asympt-contract-alon},
\ref{sec:orbit-away-from}, \ref{sec:orbit-accumul-some}
and~\ref{sec:limsup-case}. In
Subsection~\ref{sec:weak-section-expans} we prove the second
part of the statement of Theorem~\ref{mthm:negexpflow-sink}
and Theorem~\ref{mthm:posectionalexp}.  In the last
Subsection~\ref{sec:proof-pliss-lemma} we prove some
technical lemmas.


\subsection*{Acknowledgments}

This work was the result of questions posed by the students
of the PhD level course MATE51 Teoria Erg\'odica
Diferenci\'avel (Differentiable Ergodic Theory) at the
Mathematics and Statistics Institute of the Federal
University of Bahia (UFBA) at Salvador-Brazil. We thank
V. Pinheiro, P. Varandas and L. Salgado for comments and
suggestions that improved a previous version of this text,
and also the Mathematics Department at UFBA and CAPES-Brazil
for the support and basic funding of the Mathematics
Graduate Courses at MSc. and PhD. levels.  We also thank the
anonymous referees for many suggestions that helped to
improve the text.








\section{The discrete time case}
\label{sec:positive-lyapun-expo}

Here we prove
Theorems~\ref{mthm:negativexp-sink},~\ref{mthm:negexpdiffeo}
and~\ref{mthm:allnegexpfinitesinks} together with their
corollaries.

\begin{proof}[Proof of Theorem~\ref{mthm:negativexp-sink}]
  Exchanging $f$ by $f^k$ in what follows we may assume
  without loss of generality that $k=1$.  By assumption, we
  have $\zeta>0$ and a strictly increasing sequence
  $m_i\nearrow\infty$ so that
  $ \frac1{m_i}\sum_{j=0}^{m_i-1}\ln\|Df(f^{j}x)\|<-\zeta$
  as $i\nearrow\infty$.  We can now use the following.

  \begin{lemma}[Pliss Lemma; see e.g. Chapter IV.11
    in~\cite{Man87}]
    \label{le:pliss}
  Let $H\ge c_2 > c_1 >0$ and
  $\theta={(c_2-c_1)}/{(H-c_1)}$. Given real numbers
  $a_1,\ldots,a_N$ satisfying
 $
 \sum_{j=1}^N a_j \ge c_2 N \quad\mbox{and } a_j\le H
 \;\;\mbox{for all}\;\; 1\le j\le N,
 $
 there are $\ell>\theta N$ and $1<n_1<\ldots<n_\ell\le N$
 such that
$
\sum_{j=n+1}^{n_i} a_j \ge c_1\cdot(n_i-n) \;\;\mbox{for
  each}\;\; 0\le n < n_i, \; i=1,\ldots,\ell.
$
\end{lemma}

We set $c_2=-\zeta$, $c_1=c_2/2$,
$H=-\ln\inf_{x\in M}\|Df(x)\|$ and
$a_{j}=-\ln\|Df(f^{m_i-j}x)\|$ for $1< j\le m_i$. Notice
that we are inverting the summation order.

\medskip

Then for $\theta=c_2/(2H-c_2)>0$ and  $N=m_i$
Pliss Lemma ensures that there are $\ell>\theta N$ and
$1<n_1<\dots<n_\ell\le m_i$ such that for each $0\le n < n_k$
and $k=1,\ldots,\ell$
\begin{align*}
  \prod_{j=n+1}^{n_k}\|Df(f^{ m_i-j}x)\|\le e^{-c_1(n_k-n)}.  
\end{align*}
The iterates $m_i-n_k$ are \emph{reverse hyperbolic times}
for the $f$-orbit of $x$ with respect to $m_i$; similar
times were used in \cite{Man88} by Ma\~n\'e and by Liao
in~\cite{Li80}.  Pliss' Lemma ensures that there are
infinitely many reverse hyperbolic times $n_i$ along the
$f$-orbit of $x$ with respect to $m_i$ and, because
$\theta>0$, we can assume that $(m_i-n_i)\nearrow\infty$.
Consequently, if $h$ is a reverse hyperbolic time with
respect to $m_i$, then $ \|Df^{j}(f^hx)\|\le\lambda^{j} $
with $\lambda=e^{-\zeta/2}$ for all $j=1,\dots,m_i-h$.  This
uniform contractive property can be extended to a
neighborhood using the fact that $f$ is \emph{a $C^1$ map
  such that $Df$ is never the null transformation}, as
follows.

\begin{lemma}[Existence of forward contracting balls]
  \label{le:cballs}
  There exist $\de_1>0$ (depending only on $f$ and $\lambda$)
  and $\lambda_1=\sqrt\lambda\in(0,1)$ such that if $n$ is a
  reverse hyperbolic time for $x\in M$ with respect to
  $m>n$, then for every $0< j\le m-n$ there are subsets
  $V_{n+j}$ containing $f^{n+j}(x)$ such that
  $V_n=B(f^{n}(x),\de_1)$; $f^j(V_n)\subset V_{n+j}$, and
  $f^j\mid_{V_n}:V_n\to V_{n+j}$ is a
  $\lambda_1^j$-contraction.
\end{lemma}

\begin{proof}[Proof of Lemma~\ref{le:cballs}]
  We basically follow \cite[Lemma 5.2]{ABV00} adapting the
  same ideas to the present setting. Since
  $\inf_{x\in M}\|Df(x)\|>0$ we have that the map
  $\psi:M\times M\to\RR, (x,y)\mapsto \|Df(x)\|/\|Df(y)\|$
  is uniformly continuous. Hence, we can find $\delta_1>0$
  so that
  \begin{align*}
    \dist(x,y)<\delta_1, (x,y)\in M\times M
    \implies
    \frac{\|Df(x)\|}{\|Df(y)\|}\le\frac1{\lambda_1}.
  \end{align*}
  We write $x_j=f^jx$ for $j\ge0$. We construct the
  neighborhoods $V_{n+j}$ by induction on $j$. Note first that
  \begin{align*}
    y\in V_n \implies
    \|Df(y)\|\le\lambda_1^{-1}\|Df(x_n)\|\le
    \lambda_1^{-1}\lambda=\lambda_1.
  \end{align*}
  So for every pair
  $y,z\in V_n$ and a smooth curve
  $\gamma:[0,1]\to V_n$ connecting $\gamma(0)=z$
  to $\gamma(1)=y$ we have
  \begin{align*}
    \dist(fy,fz)\le|f\circ\gamma|=\int_0^1\|Df\circ\gamma\cdot\dot\gamma\|
    \le\lambda_1\int_0^1\|\dot\gamma\|=\lambda_1|\gamma|,
  \end{align*}
  where $|\gamma|$ denotes the length of the smooth curve
  $\gamma$. This shows that $f\mid_{V_n}$ is a
  $\lambda_1$-contraction.
  
  Now let us assume that $V_{n+i}$ is already defined for
  $0<i\le j<m-n-1$: $V_{n-i}$ is a set containing $x_{n+i}$
  and $f^i\mid_{V_n}:V_n\to V_{n+i}$ is a
  $\lambda_1^i$-contraction. We define
  $V_{n+j+1}=f^{n+j+1}V_n$ which contains $x_{n+j+1}$ and,
  since $\diam V_{n+i}\le\lambda_1^i\diam V_n<\delta_1$ for
  $i=1,\dots,j$, we can write for each $y_0\in V_n$
  \begin{align*}
    \|Df^{j+1}(y_0)\|
    \le
    \prod_{i=0}^j\|Df(y_i)\|
    =
    \prod_{i=0}^j\|Df(x_{n+i})\|\frac{\|Df(y_i)\|}{\|Df(x_{n+i})\|}
    \le
    \frac{\lambda^{j+1}}{\lambda_1^{j+1}}=\lambda_1^{j+1},
  \end{align*}
  so that $f^{j+1}\mid_{V_n}:V_n\to V_{n+j+1}$ is a
  $\lambda_1^{j+1}$-contraction.
\end{proof}

\begin{remark}
\label{rmk:Df0}
If we allow $Df(\bar x)\equiv0$ for some $\bar x$, then we
might have $\|Df(y_i)\|$ proportionally much larger than
$\|Df(x_i)\|$ with both $y_i,x_i$ close to $\bar x$ and the
larger factors in $\prod_i\|Df(y_i)\|$ may not be
compensated.
\end{remark}

%



\subsection{Nested contractions argument}
\label{sec:nested-contract-argu}

Since $M$ is compact and $x$ has infinitely many reverse
hyperbolic times $n_1<n_2<\dots$ with respect to
$m_1<m_2<\dots$ so that $(m_i-n_i)\nearrow\infty$, we obtain
an accumulation point $\bar x=\lim x_{n_{k_j}}$ and we
rewrite the subsequence as $x_{n_j}$ in what follows.  We
let $\xi\in(0,1)$ be such that $4\xi<1-\xi-\xi^2$ and assume
that $\dist(x_{n_k},\bar x)<\xi\delta_1$ for all $k\ge1$.
Then we choose iterates $n_j>n_2>n_1$ satisfying
$m_j-n_j>n_2-n_1$, $\lambda_1^{n_2-n_1}<1/2$ and
$\dist(x_{n_j},\bar x)<\xi^2\delta_1$; see
Figure~\ref{fig:acumulacao}.

\begin{figure}[htpb]
  \centering
  \includegraphics[width=7cm]{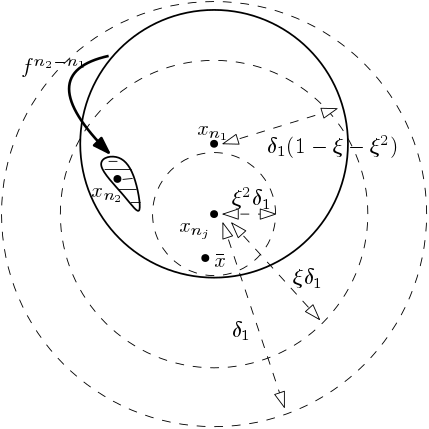}
  \caption{Relative positions of the iterates of $x$ at
    reverse hyperbolic times $n_1,n_2$ and $n_j$.}
  \label{fig:acumulacao}
\end{figure}

Then $\dist(x_{n_1},x_{n_2})<2\xi\delta_1$ and
$\dist(x_{n_1},x_{n_j})<(\xi+\xi^2)\delta_1$ and also
\begin{align*}
  B\big(x_{n_1}, \delta_1-(\xi+\xi^2)\delta_1\big)
  \subset
  B(x_{n_j},\delta_1).
\end{align*}
Moreover, since $m_j-n_j>n_2-n_1$ we can write
\begin{align*}
  f^{n_2-n_1}B(x_{n_1}, \delta_1 (1-\xi-\xi^2))
  \subset
  B(x_{n_2}, \delta_1(1-\xi-\xi^2)\lambda_1^{n_2-n_1}).
\end{align*}
We claim that 
\begin{align*}
  B(x_{n_2}, \delta_1(1-\xi-\xi^2)\lambda_1^{n_2-n_1})
  \subset
  B(x_{n_1},\delta_1(1-\xi-\xi^2)).
\end{align*}
Assuming this claim, we have the $\lambda_1^{n_2-n_1}$-contraction
\begin{align*}
  f^{n_2-n_1}\mid_{B(x_{n_1}, \delta_1 (1-\xi-\xi^2))}:
  B(x_{n_1}, \delta_1 (1-\xi-\xi^2)) \circlearrowleft
\end{align*}
and since $f$ is a continuous map, there exists a unique
fixed point $p$ for $f^{n_2-n_1}$ in this ball which is in the
basin of attraction of $p$. Since $x_{n_1}$
is in the basin of attraction of $p$, then $x_0$ belongs to
the basin of attraction of the periodic orbit $p,fp,\dots,
f^{n_2-n_1-1}p$.

To complete the proof, we prove the claim. For this it is
enough to note that
\begin{align*}
  \dist(x_{n_2},x_{n_1})+\delta_1(1-\xi-\xi^2)\lambda_1^{n_2-n_1}
  < \delta_1(1-\xi-\xi^2)
\end{align*}
if
\begin{align*}
  2\xi+(1-\xi-\xi^2)\lambda_1^{n_2-n_1}<1-\xi-\xi^2
\end{align*}
which is equivalent to
\begin{align*}
    2\xi<(1-\xi-\xi^2)(1-\lambda_1^{n_2-n_1}).
\end{align*}
This inequality is now a consequence of the choices of $\xi$
and $n_2-n_1$.
\end{proof}

\subsection{Negative Lyapunov exponents for an invariant
  probability measure}
\label{sec:negativexp-sinks}

Here we prove Corollary~\ref{cor:negativexp-sink}. Let $f$
be a $C^1$ map of $M$ and $\mu$ an $f$-invariant probability
measure satisfying $\chi(x)<0, \mu$-a.e. $x$. The
Subadditive Ergodic Theorem guarantees that
$\chi(x)=\inf_{n\ge1}\int\ln\|Df^n\|^{1/n}\,d\mu$ and so
there exists $\xi>0$ 
and we can
find $N>1$ big enough so that
$\int \ln\|Df^N\|\,d\mu<-\zeta$ for $\zeta=\xi N$.

Now we apply the following standard result.

\begin{theorem}[Ergodic Decomposition Theorem; see
  e.g. Chapter 2 in~\cite{Man87}.]
  \label{thm:ergdecomp}
  Let $f:X\to X$ be a measurable (Borelean) invertible
  transformation on the compact metric space $X$ such that
  the set of $f$-invariant probability measures $\cM(f,X)$ is
  non-empty.  Then there exists a total probability subset
  $\Sigma$ such that
\begin{itemize}
\item for every $x\in\Sigma$ the weak$^*$ limit of
  $|n|^{-1}\sum_{j=0}^{n-1}\delta_{f^j(x)}$ when $n\to\pm\infty$
  exists and equals an $f$-ergodic probability measure
  $\mu_x$;
\item for every $\mu\in\cM(f,X)$ and every $\mu$-integrable
  $\varphi:X\to\bR$, $\varphi$ is $\mu_x$-integrable for
  $\mu$-almost every $x$ and
$
        \int \varphi\, d\mu = \int \!\! \left(\int \varphi\,
          d\mu_x \right)\, d\mu(x).
$
\end{itemize}
\end{theorem}

By the ergodic decomposition of the $f^N$-invariant measure
$\mu$, we have
$
-\zeta>\int\ln\|Df^N\|\,d\mu=\int\int\ln\|Df^N\|\,d\mu_x\,d\mu(x)
$ and so the subset
$U=\{x\in M : \mu_x(\ln\|Df^N\|)<0\}$\footnote{In what
  follows, we write $\mu(\vfi)=\int\vfi\,d\mu$ for any
  integrable function $\vfi:M\to\RR$.}  satisfies $\mu
U>0$. Hence, since $\mu_x$ is $f^N$-ergodic, $x$ satisfies
\begin{align*}
  A^-_N(x)
  =
  \liminf_{n\to+\infty}\frac1n\sum_{j=0}^{n-1}\ln\|Df^N(f^{Nj}x)\|
  =
  \mu_x(\ln\|Df^N\|)
  <0.
\end{align*}
By Theorem~\ref{mthm:negativexp-sink} we conclude that
$\mu$-a.e. $x\in U$ belongs to the open basin of attraction
of some sink. Hence, for $\mu$-a.e. $x\in U$ we have that
$x\in\supp\mu$ and there exists an open neighborhood
$V_x$ of $x$ so that $\mu V_x>0$ and $V_x$ is contained in
the basin of attraction of some periodic attracting orbit
$p=p(x)\in M$ of $f^N$, which is also a sink for $f$.

This means that $\mu_y=\delta_{p(x)}$ for $\mu$-a.e.
$y\in V_x$. Since $X$ is a compact metric space and sinks
are isolated orbits, it follows that
$\mu=\tilde\mu+\sum_{i\ge1}\delta_{p_i}$ where $\tilde\mu$
is the restriction of $\mu$ to $M\setminus U$, which may be
null measure.

Finally, if we assume that $\mu$ is $f$-ergodic, then we
conclude that $\mu=\delta_p$ for some periodic sink $p$ for
$f$. The proof of Corollary~\ref{cor:negativexp-sink} is
complete.

\subsection{Negative Lyapunov exponents on total probability}
\label{sec:negative-lyapun-expo}

We now prove Theorem~\ref{mthm:allnegexpfinitesinks} and
Corollary~\ref{cor:Leballneg}.

\begin{proof}[Proof of Theorem~\ref{mthm:allnegexpfinitesinks}]
  We first claim that the assumption on $f$ implies
  \begin{align}\label{eq:allxsomek}
  \forall x\in M \; \exists k\in\ZZ^+: A^-_k(x)<0.
  \end{align}
  To prove the claim we argue by contradiction: let us
  assume that there exists $x\in M$
  satisfying $A^-_k(x)\ge0$ for all $k\in\ZZ^+$ and let
  $\mu$ be some weak$^*$ accumulation point of
  $\mu_n=n^{-1}\sum_{j=1}^{n-1}\delta_{f^{j}x}$. Since the
  assumptions on $f$ give
  \begin{align*}
    \inf_{k\ge1}\frac1k\int\ln\|Df^k\|\,d\mu=
    \lim_{n\to\infty}\frac1n\int\ln\|Df^n\|\,d\mu=
    \int\lim_{n\to\infty}\frac1n\ln\|Df^n\|\,d\mu<0,
  \end{align*}
  we obtain $\mu(\ln\|Df^k\|)<0$ for some $k\in\ZZ^+$.
  Because $\ln\|Df^k\|$ is continuous, by definition of
  weak$^*$ convergence we get
  \begin{align*}
    \mu(\log\|Df^k\|)=
    \lim_{n\to\infty}\frac1n\sum_{j=0}^{n-1}\ln\|Df^k(f^{kj}x)\|<0
  \end{align*}
  in direct contradiction with the choice of $x$. This
  contradiction proves the claim~\eqref{eq:allxsomek}.

  Now Theorem~\ref{mthm:negativexp-sink} ensures that all
  points $x$ belong to the basin of some attracting periodic
  orbit.

  We claim that there exists only one such orbit. Otherwise,
  let $\cO(p_n)$ be the (at most denumerable) collection of
  attracting periodic orbits of $f$ and let $B_n$ the
  collection of its basins, that is
  $ B_n=\{x\in M: \omega(x)=\cO(p_n)\}$.
  
  The continuity of $f$ guarantees that each $B_n$ is an
  open subset of $M$: if $f^mx\in V_n$ for some $m\in\ZZ^+$,
  where $V_n$ is a neighborhood of $p_n$ such that
  $\overline{f^{\tau_n}V_n}\subset V_n$; $f^{\tau_n}p_n=p_n$
  and $f^{\tau_n}\mid_{V_n}:V_n\to V_n$ is a contraction;
  then there exists a neighborhood $U_x$ so that
  $f^mU_x\subset V_n$ and so $U_x\subset B_n$. Clearly
  $(B_n)_n$ is a pairwise disjoint collection of subsets.

  From \eqref{eq:allxsomek} and
  Theorem~\ref{mthm:negativexp-sink} we have\footnote{From
    now on, the sum of sets denotes disjoint union.}
  $M=\sum_n B_n$.  Since these are distinct attracting
  periodic orbits, there exists a pair $m,n$ such that
  $\overline{B_n}\cap\overline{B_m}\neq\emptyset$. Otherwise,
  we would have $M=\sum_n \overline {B_n}$ and each $B_n$
  becomes simultaneously open and closed, contradicting the
  connectedness of $M$.

  Let us take $x\in \overline{B_n}\cap\overline{B_m}$. By
  \eqref{eq:allxsomek} and the previous continuity argument,
  a neighborhood $V_x$ of the point $x$ belongs to the basin
  of some attracting periodic orbit. Since
  $V_x\cap B_n\neq\emptyset\neq V_x\cap B_m$, we deduce that
  the attracting periodic orbits $\cO(p_n)$ and $\cO(p_m)$
  must be the same. This contradiction proves the claim and
  completes the proof of
  Theorem~\ref{mthm:allnegexpfinitesinks}.
\end{proof}

Now it is easy to prove
Corollaries~\ref{cor:compactnegexpperiodic}
and~\ref{cor:Leballneg}.

\begin{proof}[Proof of Corollary~\ref{cor:compactnegexpperiodic}]
  Using the notation introduced in the proof of
  Theorem~\ref{mthm:allnegexpfinitesinks}, we replace $M$ by
  the compact $f$-invariant subset $K$ and obtain
  $K=\sum_nK_n$, where each $K_n=K\cap B_n$ is relatively
  open in $K$. This open cover of $K$ must admit a finite
  subcover, so the number of sinks is finite. Moreover by
  $f$-invariance we get
  $K_n\subset\cap_{i\ge0}f^i(B_n)= \cO(p_n)$ where
  $\cO(p_n)$ is the periodic attracting orbit whose basin is
  $B_n$. So $K$ is a finite collection of sinks, completing
  the proof.
\end{proof}

\begin{proof}[Proof of Corollary~\ref{cor:Leballneg}]
  Using again the notation introduced in the proof of
  Theorem~\ref{mthm:allnegexpfinitesinks} we have, by
  Theorem~\ref{mthm:negativexp-sink}, that
  $M=\sum_n B_n, m\bmod0$. Moreover, each $B_n$ is an open
  subset of $M$. Hence $\sum_n B_n$ is an open, dense and
  full Lebesgue measure subset of $M$.
\end{proof}

\subsection{Positive Lyapunov exponents for a $C^1$
  diffeomorphism}
\label{sec:positive-lyapun-expo-1}

We are now ready to prove Theorem~\ref{mthm:negexpdiffeo}
and Corollary~\ref{cor:decomposes}.

\begin{proof}[Proof of Theorem~\ref{mthm:negexpdiffeo}]
  On the one hand, clearly $\tilde A^-_k(x)<0$ if $x$ belongs to a
  repelling periodic orbit with period $k\in\ZZ^+$.

  On the other hand, if $\tilde A^-_k(x)<0$ for some $k\in\ZZ^+$, then
  exchanging $f$ by $f^k$ we assume $k=1$ without loss of
  generality. We get $\zeta>0$ and a strictly increasing sequence
  $m_i\nearrow\infty$ so that
  $ \frac1{m_i}\sum_{j=0}^{m_i-1}\ln\|Df(f^{j}x)^{-1}\|<-\zeta$ as
  $i\nearrow\infty$. We apply Lemma~\ref{le:pliss} with $c_2=-\zeta$,
  $c_1=c_2/2$ and $H=-\ln\inf_{x\in M}\|Df(x)^{-1}\|$ and also
  $a_j=\ln\|Df(f^jx)^{-1}\|$ for $0\le j<m_i$, to obtain
  $\ell>\theta N$ with $\theta=c_2/(2H-c_2)>0$ and
  $1<n_1<\dots<n_\ell\le N$ so that
  $ \|Df^{n_k-n}(f^{n+1}x)^{-1}\| \le
  \prod_{j=n+1}^{n_k}\|Df(f^jx)^{-1}\|\le e^{-c_1(n_k-n)}, $ for each
  $0\le n < n_k$ and $k=1,\ldots,\ell$. Each $n_i$ is a hyperbolic
  time for $x$ and we can prove the following with
  $\lambda=e^{-c_1}=e^{-\zeta/2}$.

  \begin{lemma}[Existence of backward contracting balls]
    \label{le:backcontr}
    There exists $\delta_1>0$ (depending only on $f$ and
    $\lambda$) such that if $n$ is a hyperbolic time for
    $x$, then for every $0< j\le m-n$ there are
    neighborhoods $V_j$ of $f^jx$ in $M$ for which
    \begin{enumerate}
    \item $f^{n-j}$ maps $V_j$ diffeomorphically onto the
      ball of radius $\delta_1$ around $f^nx$;
    \item for $1\le j <n$ and $y, z\in V_0$,
      $ \dist(f^{n-j}(y),f^{n-j}(z)) \le
      \lambda^{j/2}\dist(f^{n}(y),f^{n}(z))$.
    \end{enumerate}
  \end{lemma}

\begin{proof}
  Just follow \cite[Lemma 5.2]{ABV00} and notice that it is
  enough to have
  \begin{align*}
    \dist(x,y)<\delta_1, (x,y)\in M\times M
    \implies
    \frac{\|Df(x)^{-1}\|}{\|Df(y)^{-1}\|}\le\frac1{\sqrt{\lambda}}
  \end{align*}
  for this proof to go through.
\end{proof}

Hence we can repeat the nested contracting argument from
Subsection~\ref{sec:nested-contract-argu} in this setting obtaining a
periodic point $p$ for some power $f^k$ of $f$ such that
$p\in B(f^nx,\delta_1)$ for some hyperbolic time $n$ of $x$. Thus, by
Lemma~\ref{le:backcontr} and since $f$ is invertible, we get
$\dist(x, f^{-n}p)\le\delta_1\lambda^{n/2}$ and we can take $n$ larger
than any predetermined quantity. We conclude that $x$ belongs to the
$f$-orbit of $p$, concluding the proof of
Theorem~\ref{mthm:negexpdiffeo}.
\end{proof}

\begin{proof}[Proof of Corollary~\ref{cor:decomposes}]
  Consider the measurable subsets $E=\{x\in M: \chi(x)<0\}$
  and $\tilde E=\{x\in M:\tilde\chi(x)<0\}$ and note that
  $E + \tilde E + M\setminus(E+\tilde E)$ is a measurable
  partition of $M$ formed by $f$-invariant
  subsets. Moreover, if $A_k^-(x)<0$ (respectively,
  $\tilde A_k^-(x)<0$) for some $k\in\ZZ^+$, then $x\in E$
  (resp., $x\in\tilde E$). In addition, if $\mu(E)>0$, then
  by Ergodic Decomposition\footnote{We write $1_A$ for the
    indicatior function: $1_A(x)=1$ if $x\in A$ and
    $1_A(x)=0$ if $x\in M\setminus A$ for $A\subset M$.}
  $\mu(1_E\cdot\chi)=\int_E\int\chi\,d\mu_x\,d\mu(x)$ and so
  $\inf_{n\ge1}\int_E\ln\|Df^n\|^{1/n}\,d\mu_x=\mu_x(1_E\cdot\chi)<0$
  for $\mu$-a.e. $x\in E$ by the Subadditive Ergodic Theorem
  applied to the invariant subset $E$. The Ergodic Theorem
  now gives $\mu_x(1_E\cdot A_k^-)\le\mu_x(1_E\cdot\chi)<0$
  for some $k=k(x)\in\ZZ^+$ (respectively,
  $\mu_x(1_{\tilde E}\cdot\tilde A_{\tilde
    k}^-)\le\mu_x(1_{\tilde E}\cdot\tilde\chi)<0$ for some
  $\tilde k=\tilde k(x)\in\ZZ^+$ if $\mu(\tilde E)>0$).

  From Theorem~\ref{mthm:negativexp-sink} we deduce that
  $\mu$-a.e. $x\in E$ belongs to the basin of a sink
  $\cO_f(p)=\{p, fp,\dots, f^{\tau-1}p\}$ for some period
  $\tau\in\ZZ^+$, and since $\mu_x$ is ergodic and $f$ an
  invertible map, we get
  $\mu_x=\tau^{-1}\sum_{j=0}^{\tau-1}\delta_{f^jp}$ (resp.,
  by Theorem~\ref{mthm:negexpdiffeo} $\mu$-a.e.
  $x\in\tilde E$ belongs to some periodic repelling orbit
  $\cO_f(q)=\{q, fq,\dots, f^{\tau-1}q\}$ and so
  $\mu_x=\tau^{-1}\sum_{j=0}^{\tau-1}\delta_{f^jq}$).

  Finally $\chi(x)\ge0$ and $\tilde\chi(x)\ge0$ for
  $x\in M\setminus(E+\tilde E)$ by construction. Hence,
  since attracting and repelling periodic orbits are
  isolated in $M$, they form an at most enumerable subset
  and so we decompose $\mu$ as in the statement of
  Corollary~\ref{cor:decomposes}.
\end{proof}


\section{The flow case}
\label{sec:flow-case}

We now prove Theorems~\ref{mthm:negexpflow-sink},
~\ref{mthm:weaknegexpFlow} and \ref{mthm:strongnegexpFlow}.
We fix $G\in\fX^1(M)$ and state a version of
Pliss' Lemma~\ref{le:pliss} for flows in the next subsection, and
then translate the assumptions of
Theorems~\ref{mthm:negexpflow-sink},
~\ref{mthm:weaknegexpFlow} and \ref{mthm:strongnegexpFlow}
in a convenient format in
Subsection~\ref{sec:linear-poincare-flow-1} to be used in
the following subsections.

\subsection{Pliss lemma for flows}
\label{sec:pliss-lemma-flows}

Following Arroyo-Hertz \cite[Theorem 3.5]{ArrHz03} we state
and prove for completeness the following version of Pliss'
Lemma for differentiable functions instead of sequences
(whose statement and proof can be found in \cite{Pl72} and
\cite{Mane83}).

\begin{theorem}
  \label{thm:plissflows}
  Given $\epsilon > 0, A,c\in\RR, c>A$, if $H : [0, T ] \to \RR$
  is differentiable, $H (0) = 0, H (T ) < cT$ and
  $c+\epsilon>\inf(H') > A$, then the set
  \begin{align*}
   \cH_\epsilon =
    \{\tau \in [0, T ] : H (s) - H (\tau ) < (c +
    \epsilon)(s - \tau ), \,\forall\tau\le s \le T\}
  \end{align*}
  has Lebesgue measure greater than $\theta T$, where
  $\theta=\epsilon/(c+\epsilon-A)$.
\end{theorem}

\begin{remark}
  \label{rmk:plissFlows}
  \begin{enumerate}
  \item This result ensures that there exists
    $\tau\in\cH_\epsilon$ such that $T-\tau>\theta T$.
  \item For given fixed $0<\eta<\bar\epsilon$, since
    $H(\tau+\eta)-H(\tau)=\int _{\tau}^{\tau+\eta}H'\ge
    A\eta$, we can write
  \begin{align*}
    H (s) - H (\tau+\eta )
    &=
      (c+\epsilon)(s-\tau)-(H(\tau+\eta)-H(\tau))
    \\
    &\le
      (c+\epsilon)(s-(\tau+\eta))+(c+\epsilon)\eta-A\eta
    \\
    &=
      \left(c+\epsilon+\eta\frac{c+\epsilon-A}{s-(\tau+\eta)}\right)
      (s-(\tau+\eta))
    \\
    &<
      (c+\hat\epsilon)(s-(\tau+\eta))
  \end{align*}
  for all $s>\tau+\eta$, where $\hat\epsilon>\epsilon$ with
  $\hat\epsilon-\epsilon$ as small as needed, if
  $\bar\epsilon$ is small enough.

  So we have $\tau+\eta\in\cH_{\hat\epsilon}$ for small
  $\eta,\hat\epsilon-\epsilon>0$ whenever
  $\tau\in\cH_\epsilon$.
\end{enumerate}

\end{remark}

We postpone the proof of Theorem~\ref{thm:plissflows} to
Section~\ref{sec:proof-pliss-lemma} and use it as a tool in
what follows.

\subsection{Linear Poincar\'e Flow and differentiability}
\label{sec:linear-poincare-flow-1}

We start the proof of Theorems~\ref{mthm:weaknegexpFlow} and
\ref{mthm:strongnegexpFlow} by expressing the assumptions in their
statements in a form suitable to apply the previous
Theorem~\ref{thm:plissflows}.

We fix
$G\in\fX^1(M)$ 
and let $P_G$ be the Linear Poincar\'e Flow of $G$.

The cocycle relation for the derivative of the flow and for
the Linear Poincar\'e Flow implies that the functions
\begin{align*}
  \Gamma(v)=\Gamma_G(v)&:\RR\times M\to\RR, \quad (t,x)\mapsto
                         \Gamma_tx=\ln\|D\phi_t(x)\cdot v\|, \quad v\in T^1_xM
                         \qand
  \\
  \psi(v)=\psi_P(v)&:\RR\times M\setminus\sing(G)\to\RR,
                     (t,x)\mapsto\phi_tx=\ln\|P^t_x\cdot v\|, \quad v\in T^1_xM, \langle v,G(x)\rangle=0,
\end{align*}
are \emph{subadditive}, where $T^1M$ is the unit tangent bundle: for
$t,s\in\RR$
\begin{align*}
  \Gamma_{t+s}(v)y
    &\le
    \Gamma_s(v)(\phi_ty)+\Gamma_t(v)y, \quad
  y\in M, \qand
\\
  \psi_{t+s}(v)x
  &\le
    \psi_s(v)(\phi_tx)+\psi_t(v)x, \quad
 x\in M\setminus\sing(G).
\end{align*}

The following result provides a sufficient condition to
ensure the existence of the time derivative of a subadditive
cocycle over a $C^1$ vector field.

\begin{lemma}
  \label{le:subdiff}
  Let $\psi:\RR\times U\to\RR$ be a subadditive function for the flow
  of $G\in\fX^1(M)$ on the invariant subset $U$ of $M$. If $\psi_0x=0$
  for all $x\in U$, $D_+(x):=\limsup_{h\to0} \frac1h\psi_hx<\infty$
  and $D_-(x):=\liminf_{h\to0} \frac1h\psi_hx<\infty$ are continuous
  functions of $x\in U$, then
  $\partial_h
  \psi_hx\mid_{h=0\pm}=D_{\pm}(x)=\lim_{h\to0\pm}\frac1h\psi_hx$ and
  the derivative exists if $D_-(x)=D_+(x)$.
\end{lemma}

\begin{proof}
  See \cite[Lemma 4.12]{AraPac2010} and its proof, where it is
  implicitly assumed that $D_-(x)=D_+(x)$ but the existence of lateral
  limits and derivatives is addressed.
\end{proof}

Now we take advantage of the fact that both
$\Gamma_G(x)=\sup_{v\in T^1_xM}\Gamma_G(v)(x)$ and
$\psi_P(x)=\sup_{v\in T^1_x M}\psi_P(v)(x)$ are continuously generated
subadditive cocycles over a $C^1$ vector field: the following results
shows that they are bounded by additive cocycles and provides useful
continuity properties of their infinitesimal generators.

\begin{lemma}\label{le:claimDx}
  Define $D_{G\pm}(x):=\pm\limsup_{h\to0\pm} \frac{\pm1}h\Gamma_hx$ and
  $D_\pm(x):=\pm\limsup_{h\to0\pm} \frac{\pm1}h\psi_hx$ and set
  $L=\sup_{x\in M}\|DG_x\|$. Then
  \begin{enumerate}
  \item $|D_{G\pm}(y)|\le L$ for all $y\in M$ and $|D_\pm(x)|\le L$
    for all $x\in M\setminus\sing(G)$;
  \item $y\in M\mapsto D_{G\pm}(y)$ and
    $x\in M\setminus\sing(G)\mapsto D_{\pm}(x)$ are continuous
    functions: in local coordinates\footnote{More precisely
      $|D_\pm(x)-D\pm(y)|\le \dim M\cdot\|\cO_xDG_x-
      D(\exp_x)^{-1}_{\exp_x^{-1}y} \circ\cO_yDG_y\circ
      D(\exp_x)_{\exp_x^{-1}y} \|$ for $y$ in the range of
      $\exp_x$.} we have
    \begin{align*}
      |D_{G\pm}(y)-D_{G\pm}(y')|&\le \dim M\cdot \|DG_y-DG_{y'}\|
                        \qand
      \\
      |D_{\pm}(x)-D_\pm(y)|&\le\dim M\cdot\|\cO_xDG_x-\cO_yDG_y\|;
    \end{align*}
  \item in addition
    \begin{align}\label{eq:lnLPFint}
      \int_0^t\!\!\!D_{G-}(\phi_sx)\,ds
      \le
      \Gamma_t(x)
      \le
      \int_0^t\!\!\!D_{G+}(\phi_sx)\,ds
      \qand
      \int_0^t\!\!\!D_-(\phi_sx)\,ds
      \le
      \psi_t(x)
      \le
      \int_0^t\!\!\!D_+(\phi_sx)\,ds;
    \end{align}
    and both $t\mapsto\Gamma_ty$ and $t\mapsto\psi_tx$ are bounded
    above and below by additive functions for any fixed $y\in M$ and
    $x\in M\setminus\sing(G)$.
  \end{enumerate}
\end{lemma}

Now we translate the assumptions of
Theorems~\ref{mthm:negexpflow-sink},~\ref{mthm:strongnegexpFlow}
and~\ref{mthm:weaknegexpFlow} using these infinitesimal generators.
We define the maps
\begin{align*}
  \wh{\Phi_t}: T^1M\cap G^\perp\to T^1M\cap G^\perp, v\mapsto
  \frac{P^tv}{\|P^tv\|}
  \qand
  \Phi_t: T^1M\to T^1M, v\mapsto \frac{D\phi_tv}{\|D\phi_tv\|}.
\end{align*}
Observe that the lateral limits exist and are equal for each
  $\psi_t(v)$ and $\Gamma_t(v)$ since
  \begin{align*}
    \ln\|P_x^{s+t}\cdot v\|
    =
    \ln\|P_{\phi_sx}^t\cdot (P_x^s\cdot v)\|
    =
    \ln\frac{\|P_{\phi_sx}^t\cdot (P_x^s\cdot v)\|}{\|P_x^s\cdot v\|}
    +
    \ln\|P_x^s\cdot v\|;
  \end{align*}
  and so we can write
  $\Gamma_{s+t}(v)=\Gamma_t(\wh{\Phi_s}v)+\Gamma_s(v)$. We similarly
  obtain additivity for $\psi_t$ with respect to $\Phi_t$.
  
  We can now define the functions
  $H_G(t,v)=\int_0^tD_G(\Phi_sv)\,ds, v\in T^1M$ and
  $H(t,v)=\int_0^tD(\wh{\Phi_s}v)\,ds, v\in T^1M\cap G^\perp$, where
  $G^\perp$ is the normal bundle to $G$ on $M\setminus\sing(G)$ and
  $D_G(v), D(w)$ are the infinitesimal generators of $\Gamma_G(v)$ and
  $\psi(w)$ respectively, for $v\in T_x^1M, w\in T_x^1M\cap
  G^\perp$. We also get equalities in~\eqref{eq:lnLPFint} with these
  generators.

\begin{lemma}\label{le:equivstat}
  The assumption of Theorem~\ref{mthm:negexpflow-sink} implies
  \begin{align}
    \label{eq:negexpflow-sink}
    \liminf_{T\to\infty}\frac1T\sup_{v\in T_x^1M}H_G(T,v)<0;
  \end{align}
  and the assumptions of Theorems~\ref{mthm:strongnegexpFlow}
  and~\ref{mthm:weaknegexpFlow} respectively imply
  \begin{align}
    \label{eq:weaknegexpflowpoint}
    \liminf_{T\to\infty}\frac1T\sup_{ v\in T_x^1M\cap G^\perp}H(T,v)<0 \qand
    \limsup_{T\to\infty}\frac1T\sup_{ v\in T_x^1M\cap G^\perp}H(T,v)<0.
  \end{align}
\end{lemma}

We postpone the proofs of these technical lemmas to
Section~\ref{sec:proof-pliss-lemma}.


\subsection{Asymptotic contraction in all directions}
\label{sec:asympt-contract-all}

We are ready to start the proof of
Theorem~\ref{mthm:negexpflow-sink}. If\footnote{There is no loss in
  generality to assume that $x\in M\setminus\sing(G)$, for otherwise
  there is nothing to prove.}  $x\in M\setminus\sing(G)$
satisfies~\eqref{eq:negexpflow-sink}, then there exists $\zeta>0$ and
$T_n\nearrow\infty$ so that
$\sup_{v\in T_x^1M}H_G(T_n,v)\le -\zeta T_n$. In addition, we observe
that
\begin{align*}
  -\zeta T_n
  >
  \sup_{v\in T^1_xM}\int_0^{T_n}D_G(\Phi_sv)\,ds
  \ge
  T_n\cdot \sup_{v\in T^1_xM} \inf_{0\le s\le T_n} D_G(\Phi_sv)
\end{align*}
and so $A=\sup_{v\in T^1_xM}\inf_{0\le t\le T_n}D_G(\Phi_sv)<-\zeta$
as required to apply Theorem~\ref{thm:plissflows} to $H_G(\cdot,v)$
with $c=-\zeta, \epsilon>0$ and fixed $v\in T_x^1M$, as long as $A$ is
a real number, which is guaranteed by Lemma~\ref{le:claimDx}. So we
apply Theorem~\ref{thm:plissflows} with $\epsilon=\zeta/4$,
$v\in T_x^1M$ to obtain times $0<\tau=\tau(v)<T_n$ so that
$(T_n-\tau)\ge \theta T_n$ with $\theta\in(0,1)$ satisfying for
$s\in[\tau,T_n]$
\begin{align}\label{eq:directional}
    \ln\|D\phi_{s-\tau}(\phi_\tau x)\Phi_\tau v\|
    =
    \int_\tau^sD_G(\Phi_uv)\,du
    \le
    -\frac{\zeta}{2} (s-\tau).
\end{align}
We can replace $v$ by any $u\in V\cap T_x^1M$ for some open
neighborhood $V$ of $v$ in the unit sphere at $T_xM$. By compactness
of the unit sphere, we obtain a finite cover $V_1,\dots, V_k$ of
$T_x^1M$ with associated times $\tau_1,\dots,\tau_k$ so that
$T_n-\tau_i\ge\theta T_n$ and $u\in V_i$
satisfy~\eqref{eq:directional} in the place of $v$ for
$T_n-\tau_i\le s\le T_n$, for each $i=1,\dots,k$. Hence for
$\tau_n=\min_{i=1,\dots,k} \tau_i$ we obtain ~\eqref{eq:directional}
for all $v\in T_x^1M$ and $T_n-\tau_n\le s\le T_n$, and also
$T_n-\tau_n\ge\theta T_n$.

In particular, since $\Phi_t$ is a bijection from $T_x^1M$ to
$T^1_{\phi_t x} M$, we have that $0$ is a reverse hyperbolic time of
$f=\phi_1$ for the point $x(\tau_n)=\phi_{\tau_n}x$ with respect
to\footnote{Here $[t]=\sup\{n\in\ZZ^+:n\le t\}$ is the integer part of
  $t\in\RR$.}  $[T_n-\tau_n]$, i.e.
$\|Df^k(x(\tau_n))\|\le\prod_{j=0}^{k-1}\|Df(f^jx(\tau_n))\|\le
e^{-\zeta k/2}, 1\le k\le[T_n-\tau_n]$.

Applying Lemma~\ref{le:cballs} we obtain
$\delta_1,\lambda_1$ and $W_n=B(x(\tau_n),\delta_1)$ such
that $f^j\mid_{W_n}: W_n\to f^jW_n=W_{n+j}$ is a
$\lambda_1^j$-contraction for $0<j\le[T_n-\tau_n]$.

We are in the exact same setting of
Subsection~\ref{sec:nested-contract-argu}. Hence, we can
find $m\in\ZZ^+$ so that $f^m$ has a contracting fixed point
$p$ whose basin contains $x(\tau_n)$ for some large
$n\in\ZZ^+$. Thus $p=\sigma\in\sing(G)$ is a hyperbolic
attracting singularity (a sink) for the vector field $G$ and
$\phi_tx\xrightarrow[t\to\infty]{}\sigma$. This completes
the proof of the first part of the statement of
Theorem~\ref{mthm:negexpflow-sink}.


\subsection{Asymptotic sectional contraction along a
  trajectory}
\label{sec:asympt-contract-alon}

If the trajectory of $x\in M\setminus\sing(G)$ satisfies the
left hand side of~\eqref{eq:weaknegexpflowpoint}, then there exists
$c=-\zeta$ and $T_n\nearrow\infty$ so that
 $\sup_{v\in T_x^1M\cap G^\perp}H(T_n,v)\le c T_n$ and moreover
\begin{align*}
  c T_n>
  \int_0^{T_n}D(\wh{\Phi_s}v)\,ds
  \ge
  T_n\cdot \sup_{v\in T_x^1M\cap G^\perp}\inf_{0\le s\le T_n}
  D(\wh{\Phi_s}v)
  = T_n\cdot A
\end{align*}
and so $A<c$. We apply Theorem~\ref{thm:plissflows} to $H(\cdot,v)$
with $\epsilon=-c/2$ and fixed $v\in T_x^1M\cap G^\perp$ to obtain,
reasoning as in the previous subsection, times $0<\tau_n<T_n$ so that
$(T_n-\tau_n)\nearrow\infty$ satisfying for $s\in[\tau_n,T_n]$
\begin{align}
  \label{eq:LPFcontraction}
    \ln\|P^{s-\tau_n}_{\phi_{\tau_n}x}\|
    =
    \sup_{v\in T_x^1M\cap G^\perp}\int_{\tau_n}^sD(\wh{\Phi_u}v)\,du
    \le
    -\frac{\zeta}{2} (s-\tau_n).
\end{align}
We say that $\tau_n$ is an $e^{-\zeta/2}$-reverse hyperbolic
time with respect to $T_n$.

We divide the proof of Theorem~\ref{mthm:weaknegexpFlow} in
two main cases presented in the following
Subsection~\ref{sec:orbit-away-from}, for trajectories not
accumulating any equilibrium;
and Subsection~\ref{sec:orbit-accumul-some} for trajectories which
accumulate some equilibrium.

Afterwards, we complete the proof of
Theorem~\ref{mthm:strongnegexpFlow} in
Subsection~\ref{sec:limsup-case}.

\begin{remark}
  \label{rmk:slidehyptime}
  For $\tau_n<T_n$ as in~\eqref{eq:LPFcontraction}, any $t>\eta>0$ and
  assuming without loss of generality that $\zeta<4L$, we get
  $\int_{\tau_n+\eta}^tD(\wh{\Phi_s}v)\,ds =
  \int_{\tau_n}^tD(\wh{\Phi_s}v)\,ds -
  \int_{\tau_n}^{\tau_n+\eta}D(\wh{\Phi_s}v)\,ds \le -\zeta t/2+L\eta
  $ for each $v\in T_x^1M\cap G^\perp$ (recall that $|D|\le L$) which
  is bounded by $(-\zeta/2+L\eta/t)t <- \zeta t /4 < -\zeta(t-\eta)/4$
  since $L\eta/t<\zeta/4 \iff \eta<t\zeta/4L$ and
  $\eta<\eta\zeta/4L<t\zeta/4L$. This shows that if $\tau_n$ is an
  $e^{-\zeta/2}$-reverse hyperbolic time w.r.t.  $T_n$, then
  $\ln\|P_{\phi_{\tau_n+\eta}}^t\|=\sup\int_{\tau_n+\eta}^tD(\wh{\Phi_s}v)\,ds
  <-\zeta(t-\eta)/4$ and so any $s\in(\tau_n,T_n)$ becomes a
  $e^{-\zeta/4}$-reverse hyperbolic time w.r.t. $T_n$.
\end{remark}


\subsection{Trajectory away from equilibria}
\label{sec:orbit-away-from}

First, we assume that $\omega_G(x)\cap\sing(G)=\emptyset$
and that $\sing(G)$ is a finite subset, so that there exists
$d_0>0$ such that
$\dist(\phi_tx,\sing(G))\ge d_0, \forall t\ge0$ and also
$\dist(\omega_G(x),\sing(G))\ge d_0$.

We show that~\eqref{eq:LPFcontraction} implies that the
flow contracts distances uniformly in the transverse
direction to the vector field along longer and longer orbit
segments of the positive orbit of $x$. Compacteness of $M$ then
guarantees, by an argument similar to the one presented in
Section~\ref{sec:positive-lyapun-expo}, the existence of a
Poincar\'e section of the flow, together with a neighborhood
of a hitting point of the orbit of $x$, which is sent inside
itself by some Poincar\'e return map. This provides a sink
for that Poincar\'e return map which, as is well-known, gives a
periodic attracting orbit for the flow containing $x$ in its
basin.

\subsubsection{Forward sectional contracting balls}
\label{sec:forward-section-cont}

The uniform bound on the distance away from equilibria
ensures that there exists $0<\rho\le d_0$ such that for each
$y\in\cO^+_G(x)\cup\omega_G(x)$\footnote{We write
  $\cO^+_G(x)=\{\phi_tx: t\ge0\}$ and note that both
  $\omega_G(x)$ and $\sing(G)$ are compact.} we can
construct a Poincar\'e cross-section of $G$ through $y$ as
\begin{align}
  \label{eq:Psectionormal}
  S_y=\exp_y(B(0,\rho)\cap G(y)^\perp),
\end{align}
where $\exp_y:T_yM\to M$ is the standard exponential map
induced by the Riemannian structure on $M$; $B(0,\rho)$ is
the $\rho$-neighborhood of the origin in the tangent space
$T_yM$ with the distance induced by
$\|\cdot\|_y=\langle\cdot,\cdot\rangle^{1/2}$; and
$G(y)^\perp$ is the subspace of $T_yM$ orthogonal to
$G(y)$. We also write
$S_y(\xi)=\exp_y(B(0,\xi\rho)\cap G(y)^\perp)$ for
$\xi\in(0,1]$.

By uniform continuity of
$\eta:(M\setminus B(\sing(G),d_0)))^2 \to\RR,
(z,w)\mapsto\|P^1_z\|/\|P^1_w\|$ together with the
subadditivity of $\psi$, we can find $\xi_0>0$ so that
\begin{align}
  \label{eq:unifcontquotient}
  \phi_{[-2,2]}S_y(\xi_0)\cap\sing(G)=\emptyset
  \;\text{  and   }\;
  z,w\in S_y(\xi_0), 0\le t\le1 \implies
  \frac{\|P^t_z\|}{\|P^t_w\|}\le e^{\zeta/4}=\lambda_1^{-1}.
\end{align}
Similarly to Section~\ref{sec:positive-lyapun-expo}, we
write $x(t)=\phi_tx$ for $t\in\RR$ in what follows.
\begin{proposition}[Existence of forward sectional
  contracting balls]
  \label{pr:tubecontraction}
  Let $\tau_n<T_n$ be the pair of strictly increasing
  sequences obtained before
  satisfying~\eqref{eq:LPFcontraction}. For every
  $\delta_0>0$ there exists $\xi_0>0$ satisfying
  \eqref{eq:unifcontquotient} such that, if
  $\dist(\phi_tx,\sing(G))\ge d_0, \forall
  t\in[\tau_n,T_n]$, then for each $s\in(\tau_n,T_n]$ there
  exists a $C^1$ smooth well-defined diffeomorphism with its
  image $R_s:S_{x(\tau_n)}(\xi_0)\to S_{x(s)}(\xi_0)$ such
  that $R_s$ is a Poincar\'e map, $R_s(x(\tau_n))=x(s)$ and
  $R_s$ is an $e^{-\frac{\zeta}4(s-\tau_n)}$-contraction.
\end{proposition}

This result is the analogous to Lemma~\ref{le:cballs} in the
flow setting with $\xi_0\rho$ playing the role of
$\delta_1$; see Figure~\ref{fig:secontrball}.

\begin{figure}[htpb]
  \centering
  \includegraphics[width=10cm]{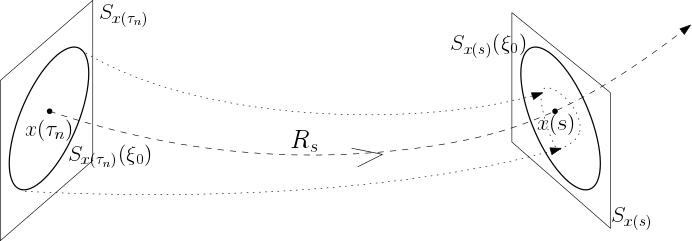}
  \caption{Sketch of a sectional contracting ball.}
  \label{fig:secontrball}
\end{figure}

\begin{proof}[Proof of Proposition~\ref{pr:tubecontraction}]
  Using $\dist(\phi_{[\tau_n,T_n]}(x),\sing(G))\ge d_0>0$ we
  can find $\xi_0$ satisfying \eqref{eq:unifcontquotient} as
  explained before the statement of the Proposition.
  
  We note that, by construction,
  $T_{\tau_n}S_{x(\tau_n)}(\xi_0)=G(x(\tau_n))^\perp$ so,
  from the choice of $\xi_0$ and for $s=\tau_n+1$ there
  exists a well-defined Poincar\'e map $R_s:W_s\to S_{x(s)}$
  from a neighborhood $W_s$ of $x(\tau_n)$ in
  $S_{x(\tau_n)}(\xi_0)$.  We also have
  \begin{align*}
    DR_s(x(\tau_n))
    &=
      \cO_{T_{x(s)}S_{x(s)}}\circ
      D\phi_{s-\tau_n}(x(\tau_n))\mid_{T_{x(\tau_n)}S_{x(\tau_n)}(\xi_0)}
    \\
    &=
      \cO_{G(x(s))^\perp}\circ
      D\phi_{s-\tau_n}(x(\tau_n))\mid_{G(x(\tau_n))^\perp}
      =
      P^1_{\phi_{\tau_n}x}.
  \end{align*}
  This together with~\eqref{eq:unifcontquotient} and
  Remark~\ref{rmk:slidehyptime} ensures that
  $\|DR_s(z)\|< e^{-\frac{\zeta}4(s-\tau_n)}=\lambda_1$ for all
  $z\in W_s$. Then the map $R_s$ is a diffeomorphism with its image
  which contracts distances at a rate $\lambda_1$.

  We claim that we may take $W_s=S_{x(s)}(\xi_0)$. To prove this,
  we fix a direction $v\in G(x(\tau_n))^\perp$ with
  $\|v\|=1$ and consider the set
  \begin{align*}
    E=\{t\in(0,\xi_0)
    &:
    R_s\big(\exp_{x(\tau_n)}(\xi v)\big)\in S_{x(s)}(\lambda_1)
    \quad\text{and}
    \\
    &\|DR_s\big(\exp_{x(\tau_n)}(\xi v)\big)\|< \lambda_1,
      \forall 0\le \xi\le t\}.
  \end{align*}
  Clearly $E\subset(0,\xi_0)$ and we have already shown that
  $\sup E>0$. We note that the claim follows if we prove
  that $\sup E=\xi_0$. Indeed, since the unit vector $v$ was
  arbitrarily chosen in $G(x(\tau_n))^\perp$, then $\sup
  E=\xi_0$ implies that the Poincar\'e map $R_s$ is
  well-defined and a $\lambda_1$-contraction on the whole of
  $S_{x(\tau_n)}(\xi_0)$.

  To prove that $\sup E=\xi_0$ we argue by contradiction: let
  us assume that $0<\alpha=\sup E<\xi_0$. Then the curve
  $\gamma(t)=\exp_{x(\tau_n)} (t v), t\in[0,\alpha]$ is sent
  to a curve $R_s(\gamma)\subset S_{x(s)}(\xi_0)$ with
  length\footnote{By construction $\gamma$ is a curve with
    unit speed.}
  \begin{align*}
    |R_s\circ\gamma|
    &=
    \int_0^\alpha\|DR_s\circ\gamma\cdot\dot\gamma(t)\|\,dt
    \le
      \lambda_1\int_0^\alpha\|\dot\gamma(t)\|\,dt=\lambda_1\alpha
      <\lambda_1\xi_0.
  \end{align*}
  Hence, on the one hand, we have for each
  $0<t<\alpha$
  \begin{align*}
  \dist_S(R_s(\gamma(t)),x(s))\le|R_s\circ\gamma|\le\lambda_1\xi_0<\xi_0,
  \end{align*}
  where $\dist_S$ is the induced distance on $S_{x(s)}(\xi_0)$ by
  the Riemannian distance of $M$. But
  $\phi_{[-2,2]}S_{x(\tau_n)}(\xi_0)$ is a flow box, thus
  $\phi_{[-2,2]}x(\gamma(\alpha))\cap
  S_{x(s)}(\xi_0)=\lim_{t\to\alpha}R_s(\gamma(t))$ and we can
  extend $R_s$ from $\gamma([0,\alpha))$ to
  $\gamma(\alpha)$.  Then
  $\dist_S(R_s(\gamma(\alpha)),x(s))\le\lambda_1\xi_0<\xi_0$
  and this enables us to use the flow box again to extend $R_s$ to
  a neighborhood of $\gamma(\alpha)$ in
  $S_{x(\tau_n)}(\xi_0)$. This shows that there is $t\in E $ with
  $t>\alpha$ and this contradiction completes the proof of
  claim that $\sup E=\xi_0$.

  We observe that the argument above is valid for any
  $s\in(\tau_n,\tau_n+1]$ replacing the rate $\lambda_1$ by
  $\lambda_1^{s-\tau_n}$.

  Let now $s\in(\tau_n, T_n]$ be given and let us write
  $s-\tau_n=k+\xi$ with $k\in\ZZ^+$ and $\xi\in[0,1)$. Then
  we divide the interval $[\tau_n,s]$ into
  $\{[\tau_n+i,\tau_n+i+1)\}_{i=0,\dots,k-1}$ together with
  $[\tau_n+k,s]$ and consider the Poincar\'e maps
  $R_i:S_{x(\tau_n+i)}(\xi_0)\to S_{x(\tau_n+i+1)}(\xi_0)$
  for $i=0,\dots,k-1$ and
  $R_k:S_{x(\tau_n+k)}(\xi_0)\to S_{x(T_n)}(\xi_0)$.

  Finally, since the image of $R_i$ is contained in the
  $\lambda_1\rho$-ball around $x(\tau_n+i+1)$ in
  $S_{x(\tau_n+i+1)}(\xi_0)$ for $i=0,\dots,k-1$ and the image of
  $R_k$ is inside the $\lambda_1^{\xi}$-neighborhood of
  $x(s)$ in $S_{x(s)}(\xi_0)$, then the composition
  $R_k\circ R_{k-1}\circ\dots\circ R_0$ is well-defined and
  a $\lambda_1^{k+\xi}$-contraction from $S_{x(\tau_n)}(\xi_0)$ to
  $S_{x(s)}(\xi_0)$. Since
  $\lambda_1^{k+\xi}=e^{-\frac{\zeta}4(s-\tau_n)}$, the
  proof of the proposition is complete.
\end{proof}


\subsubsection{Infinitely many contracting flow boxes with
  arbitrary long size and uniform domains}
\label{sec:infinit-many-contrac}

We follow the same strategy as in
Subsection~\ref{sec:nested-contract-argu} replacing $\delta_1$-balls
by $\xi_0\rho$-neighborhoods on local cross-sections: by compactness
we fix a subsequence of $\tau_n$, which we denote by the same letters,
such that $x(\tau_n)\xrightarrow[n\to\infty]{}\bar x\in\omega(x)$ and
$(T_n-\tau_n)\nearrow\infty$ together with $S_{\bar x}=S_{\bar x}(1)$
and (using Proposition~\ref{pr:tubecontraction}) a collection of
Poincar\'e maps with domains in a neighborhood $S_{\bar x}(\xi_0)$ of
$\bar x$ in $S_{\bar x}$, as follows.

We start by fixing the local Poincar\'e map\footnote{That
  is, $\Phi(\phi_t x)=x$ for all $x\in S(\xi_0)$ and $-2\le t\le2$.}
$\Phi:\phi_{[-2,2]}S_{\bar x}(\xi_0)\to S_{\bar x}(\xi_0)$ and
$\xi\in(0,1)$ such that $4\xi<1-\xi-\xi^2$. We assume
without loss of generality that
$\Phi x(\tau_k)\subset S_{\bar x}(\xi\xi_0/2)$ for all $k\ge1$ and
that $\Phi x(\tau_k)=x(\tau_k+\eta_k)$ for some
$0<\eta_k<\bar\epsilon$ and $\bar\epsilon>0$ small, since
$\omega(x)$ is an invariant set under the action of the
flow.  Then we choose $j$ so that
$T_j-\tau_j>\tau_2-\tau_1, \lambda_1^{\tau_2-\tau_1}<1/2$
and $\dist(\Phi x(\tau_j),\bar x)<\xi^2\xi_0\rho/2$; see
Figure~\ref{fig:acumulacao} again setting
$x_{n_i}=\Phi x(\tau_i+\eta_i), i=1,2,j$.

Note that, from Remark~\ref{rmk:plissFlows}, the times
$\tau_i+\eta_i$ also satisfy the conclusion of
Proposition~\ref{pr:tubecontraction}, since we may take
$\bar\epsilon>0$ as small as needed.

Now we just repeat the arguments in
Subsection~\ref{sec:nested-contract-argu} with
$\delta_1=\xi_0\rho/2$ and
$f^{n_2-n_1}=\Phi\circ\phi_{\tau_2+\eta_2-(\tau_1-\eta_1)}$
to obtain an attracting fixed point $p$ for this last map
whose basin in $S_{\bar x}(\xi_0)$ contains $x_{n_1}$. Then
the orbit $\cO_G(p)$ is periodic and since $p$ is a sink for
$f^{n_2-n_1}$, we conclude that $\cO_G(p)$ is a periodic
(hyperbolic) sink for $G$ and $x$ belongs to its basin of
attraction.

This completes the proof of
Theorem~\ref{mthm:weaknegexpFlow} in this case.


\subsection{Trajectory accumulating some equilibrium}
\label{sec:orbit-accumul-some}

Alternatively, we assume that the orbit of $x$ accumulates
some singularity, that is $\sigma\in\omega(x)\cap\sing(G)$
and, from now on, we assume that each element of $\sing(G)$
is hyperbolic. Then
\begin{itemize}
\item either $\omega(x)=\{\sigma\}$ and so $x(t)\to\sigma$
  when $t\to+\infty$ and
  \begin{itemize}
      \item if $\sigma$ is a sink, then $x$ belongs to its
        basin and we have nothing to prove;
      \item if $\sigma$ is a source, then $\phi_tx\in U$ for
        some small neighborhood $U$ of $\sigma$ and
        arbitrarily large values of $t>0$. Hence
        $x\in\phi_{-t}U$ and $\diam\big(\phi_{-t}U\big)\to0$
        when $t\nearrow\infty$ so $x=\sigma$, a
        contradiction.  So we are left with
  \item $\sigma$ is a hyperbolic saddle and $x$ belongs to its stable
    manifold.
  \end{itemize}
\item or $\omega(x)\supsetneq\{\sigma\}$ and then $\sigma$
  is again a hyperbolic saddle equilibrium, since
  \begin{itemize}
  \item if $\sigma$ is a sink, then because
    $x(t)\in W^s_{G}(\sigma)$ for some $t>0$, we conclude that
  $x(t)\to\sigma$ when $t\to+\infty$ and so
  $\omega(x)=\{\sigma\}$, a contradiction; otherwise
\item $\sigma$ is a source, and then $x=\sigma$ which is a
  contradiction again.
  \end{itemize}
\end{itemize}
Next we argue that such accumulation can only happen if
$\sigma$ is a codimension $1$ saddle, completing the proof
of Theorem~\ref{mthm:weaknegexpFlow}.

\subsubsection{Trajectory in the stable manifold of some
    equilibrium}
\label{sec:case-xttosigm}

In case $x(t)\to\sigma$ when $t\nearrow\infty$, then $x\in
W^s(\sigma)$ and we prove the following for later use.

\begin{lemma}
  \label{le:stablemanifold}
  Let $\sigma\in\sing(G)$ be a hyperbolic equilibrium and
  $q\in W^s(\sigma)\setminus\{\sigma\}$ such that
  $\liminf_{t\to\infty}\ln\|P_q^t\|^{1/t}<0$. Then $\sigma$ is a sink.
\end{lemma}

Applying the lemma shows that $\omega(x)=\{\sigma\}$ can
only happen if $\sigma$ is a sink.

We need the following consequence of Gronwall's Inequality
in several arguments in what follows, so we state here for
later use and present a proof in
Section~\ref{sec:proof-pliss-lemma}.

\begin{lemma}
  \label{le:Gronwall}
  Let $q_n\in M$, $\sigma\in\sing(G)$ and $t>0$ be such
  that\footnote{As usual $\dist(A,B)=\inf\{\dist(a,b):a\in
    A, b\in B\}$ for $A,B\subset M$.}
  $\dist(\phi_{[0,t]}q_n,\sigma)
  \xrightarrow[n\to\infty]{}0$. Then
$
\|P_{q_n}^t -\cO_{\phi_tq_n} e^{tDG_\sigma}\|
\le
\|D\phi_t(q_n)-e^{tDG_\sigma}\|
\le\bar\delta_n t e^{Lt}
$
  where
      \begin{align*}
  \bar\delta_n=\sup_{0\le s\le t}&
  \|\cO_{\phi_s q_n}DG_{\phi_s q_n}-\cO_{\phi_s q_n}DG_{\sigma}\|
    \le
  \sup_{0\le s\le t}\|DG_{\phi_s q_n}-DG_{\sigma}\|
  \xrightarrow[n\to\infty]{}0.
\end{align*}
\end{lemma}

Now we can present the proof of the previous lemma.

\begin{proof}[Proof of Lemma~\ref{le:stablemanifold}]
  Since
  $\liminf_{t\to\infty} \sup_{v\in T_q^1M\cap G^\perp}H(t,v)/t<0$ for
  $H(t,v)=\int_0^tD(\wh{\Phi_s}v)\,ds$, we can find $\zeta>0$ and
  reverse hyperbolic times $\tau_n$ associated to $T_n\nearrow\infty$
  so that $T_n-\tau_n>\theta T_n$ as in~\eqref{eq:LPFcontraction}.

  We take $\tau_n>0$ large enough so that
  $\dist(\phi_{[\tau_n,T_n]}q,\sigma) \xrightarrow[n\to\infty]{}0$ and
  for any given fixed $0<t<T_n-\tau_n$ we apply
  Lemma~\ref{le:Gronwall} with $q_n=q(\tau_n)=\phi_{\tau_n}q$ to get
  $ \|P_{q(\tau_n)}^t -\cO_{x(\tau_n+t)} e^{tDG_\sigma}\| \le
  \bar\delta_n t e^{Lt} \xrightarrow[n\to\infty]{}0$.

Finally, since $\|P_{q(\tau_n)}^t\|\le e^{-\zeta t/2}$ for
all $0<t<T_n-\tau_n$, by the definition of $\tau_n$ as a
reverse hyperbolic time, we conclude that for any given
fixed $t>0$, non-zero vectors in $G(q(\tau_n))^\perp$ are
contracted by $e^{tDG_\sigma}$ at a positive rate. But for
any vector $v\in E^u_\sigma$ we have
$\|\cO_{q(\tau_n+t)}v\|\ge\|v\|/2$ for all $n$ large enough
and so by invariance of $E^u_\sigma$ we deduce
\begin{align*}
\frac12\|e^{tDG_\sigma}v\|
  \le
  \|\cO_{x(\tau_n+t)} e^{tDG_\sigma}v\|
  \le
  \bar\delta_n\|v\|+\|P_{q(\tau_n)}^tv\|
  \le
  (\bar\delta_n+e^{-\zeta t/2})\|v\|
\end{align*}
and since $t>0$ is arbitrary, we conclude that $v=\vec0$,
that is, $E^u_\sigma=\{\vec0\}$ and $\sigma$ is a sink.
\end{proof}

\subsubsection{Trajectory accumulates but does not converge
  to an equilibrium}
\label{sec:case-omegaxs}

If $\omega(x)\supsetneq\{\sigma\}$,
then we again separate the argument into
different cases, as follows.

\begin{enumerate}
\item \emph{The orbits segments $x([\tau_n,T_n])$ are away
    from $\sing(G)$.}

  If there exists a subsequence $n_k\nearrow\infty$ such
  that the family of orbit segments
  $\{ x([\tau_{n_k},T_{n_k}])\}_k$ does not accumulate
  $\sing(G)$, then we can argue just as in the previous
  Subsection~\ref{sec:orbit-away-from}. That is, we consider
  $\bar x$ an accumulation point of $x(\tau_{n_k})$ and the
  cross-section $S_{\bar x}$ with a size $\rho$ given by at
  most
  $\inf_k\{\dist\big(x([\tau_{n_k},T_{n_k}]),\sing(G)\big)\}>0$,
  and repeat the same reasoning in
  Subsection~\ref{sec:infinit-many-contrac} to obtain a sink
  in $S_{\bar x}$ for some Poincar\'e return map. This is a
  contradiction with the assumption that
  $\omega(x)\neq\{\sigma\}$ and we conclude that this case cannot
  happen.
  
\item \emph{Alternatively}, since $\sing(G)$ is finite,
  \emph{there exists
    $\sigma_0\in\sing(G)$ and
    $s_{n_k}\in[\tau_{n_k},T_{n_k}]$ so that
    $x(s_{n_k})\to\sigma_0$.}
\end{enumerate}
In what follows we reindex the sequences to $\tau_k,T_k$ and
$s_k$ to simplify the notation. We note that $\sigma_0$ must
be a hyperbolic saddle; for otherwise $\sigma_0$ would be a
sink and then $\omega(x)=\{\sigma_0\}$.

\begin{lemma}[convergence to stable manifold of
      $\sigma_0$]
  \label{le:accstablesink}
  Let $\tau_n<T_n$ be such that $\tau_k\nearrow\infty$,
  $(T_k-\tau_k)\nearrow\infty$ and satisfy the left hand side
  of~\eqref{eq:LPFcontraction}. Assume that there exists a
  hyperbolic saddle $\sigma_0\in\sing(G)$ and
  $s_n\in[\tau_k,T_k]$ so that $x(s_k)\to\sigma_0$. If there
  exists $q\in M\setminus\sing(G)$ and (perhaps for a
  subsequence) $x(\tau_k)\to q$, then
  $q\in W^s(\sigma_0)\setminus\{\sigma_0\}$ and for $t\ge0$
  we have $\|P^t_q\|\le e^{-\frac{\zeta}{4}t}$.
\end{lemma}

\begin{figure}[htpb]
  \centering
  \includegraphics[width=13cm]{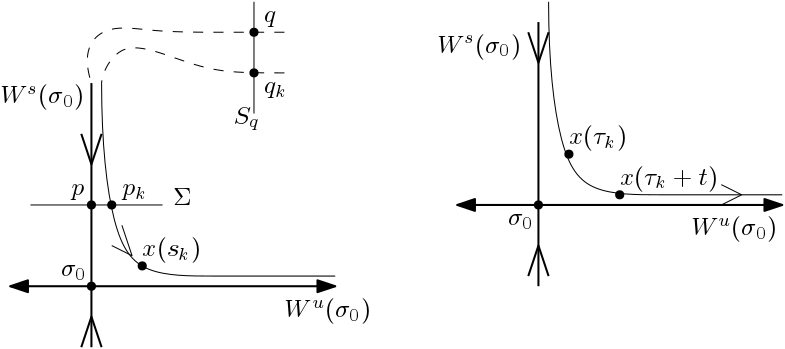}
  \caption{Relative positions of $p, p_k=x(\tau_k)\in\Sigma$ and
    $x(s_k)$ close to $\sigma_0$ on the left hand side; and
    of $x(\tau_k), x(\tau_k+t)$ and $\sigma_0$ on the right hand side.}
  \label{fig:acumulasigma}
\end{figure}

\begin{proof}
  We observe that we can assume without loss of generality
  that the segment $x([\tau_k,s_k])$ does not accumulate
  $\sing(G)\setminus\{\sigma_0\}$. For otherwise, because
  $\sing(G)$ is finite, we would replace $s_k$ by another
  sequence $\tau_k<s_k'<s_k$ satisfying this property.

  Then we can find a cross-section $\Sigma$ of $G$ at
  $p\in W^s(\sigma)\setminus\{\sigma\}$ close to $\sigma$ so
  that the segment $x([\tau_k,s_k])$ crosses $\Sigma$ at a
  point $p_k=x(\tau_k+\upsilon_k)$ and
  $p_k\xrightarrow[k\to\infty]{}p$; see the left hand side
  of Figure~\ref{fig:acumulasigma}.

  Hence, we may apply Proposition~\ref{pr:tubecontraction}
  to the pair $\tau_k,\tau_k+\upsilon_k$ after choosing
  $S_q(\xi_0)$ the cross-section of $G$ through $q$ with
  uniform size, where the value of $\xi_0$ depends on the
  distance between $\{ x([\tau_k,\tau_k+\upsilon_k]) \}_k$
  and $\sing(G)$. Since
  $q_k=x(\tau_k+\eta_k)=\Phi x(\tau_k)\in S_q(\xi_0)$ is
  such that $q_k\to q$ (and so $\eta_k\to0$), we have that
  $q_k\in S_q(\xi_0/2)$ for all big enough $k$, and obtain a
  Poincar\'e map $R:S_q(\xi_0/2)\to\Sigma$ so that
  $p_k=Rq_k$.  Hence, $p=Rq$ and thus $q\in W^s(\sigma)$ as
  claimed.

  Moreover,
  $\sup_{0\le t\le\upsilon_k}
  |\phi_tq-x(\tau_k+t)|\xrightarrow[k\to\infty]{}0$ and also
  by construction we obtain that for any given
  $\epsilon_0>0$ there exists $m\in\ZZ^+$ such
  that\footnote{We can assume that all linear maps are
    comparable in local charts given by the exponential
    map.}
  $\sup_{0\le t\le\upsilon_k}
  \|\cO_{\phi_tq}DG_{\phi_tq}-\cO_{x(\tau_k+t)}DG_{x(\tau_k+t)}\|<\zeta/4$
  for all $k>m$. Thus for any unit vector $v$ not parallel to $G(q)$
  and $G(x(\tau_k))$ we obtain from the expressions for $H(v,t)$ at
  $q$ and $x(\tau_k)$ that
  $ \ln\frac{\|P^t_qv\|}{\|P^t_{x(\tau_k)}v\|} \le
  \int_0^t|D(\wh{\Phi^q_s}v)-D(\wh{\Phi_s^{x(\tau_k)}}v)|\,ds \le
  \frac{\zeta}4t $, where $\wh{\Phi_s^z}v=\frac{P_z^sv}{\|P_z^sv\|}$
  for any $z\in M\setminus\sing(G)$ and $v\in T_z^1M$ not parallel to
  $G(z)$. This ensures that $\ln\|P^t_qv\|\le-\zeta t /4$ for all
  $v\in T_q^1M\cap G^\perp$ which implies last inequality in the
  statement of the lemma for $0\le t\le\upsilon_k$.

  Finally, note that we may take $\Sigma=\Sigma_k$ closer to
  $\sigma_0$ and obtain $\upsilon_k\nearrow\infty$. This
  completes the proof of the lemma.
\end{proof}

\subsubsection{Conclusion of the proof of
  Theorem~\ref{mthm:weaknegexpFlow}}
\label{sec:conclus-proof-theore}

Now we subdivide the argument to conclude the proof of
Theorem~\ref{mthm:weaknegexpFlow} into the following cases
according to the accumulation points of $x(\tau_k)$.

\begin{description}
\item[Case $x(\tau_k)\to q\notin\sing(G)$] from
  Lemma~\ref{le:accstablesink} we have
  $q\in W^s(\sigma_0)\setminus\{\sigma_0\}$ and also
  $\|P_q^t\|\le e^{-\frac{\zeta}4 t}, t\ge0$.  From
  Lemma~\ref{le:stablemanifold} we conclude that $\sigma_0$
  is a sink.    
  \item[Otherwise, $x(\tau_k)\to \sigma\in\sing(G)$] clearly
    $\sigma$ is again a saddle.  By the local linearization
    given by the Hartman-Grobman Theorem, for any given
    fixed $t>0$ we have
    $\dist\big( x(\tau_k,T_k],\sigma)\big)
    \xrightarrow[k\to\infty]{}0$; see e.g. \cite{PM82} and
    the right hand side of Figure~\ref{fig:acumulasigma}.

   We now use Lemma~\ref{le:Gronwall} to obtain
$ \|P_{x(\tau_k)}^t -\cO_{x(\tau_k+t)} e^{tDG_\sigma}\| \le
\bar\delta_k t e^{Lt} \xrightarrow[k\to\infty]{}0 $.
  
Now, since $\|P_{x(\tau_k)}^t\|\le e^{-\zeta t/2}$ for all
$0<t<T_k-\tau_k$ by the definition of $\tau_k$ as a reverse
hyperbolic time, then non-zero vectors in
$G(x(\tau_k))^\perp$ are contracted by $e^{tDG_\sigma}$ at a
positive rate for any given fixed $t>0$. This shows that
$\sigma$ is a saddle with contracting direction of dimension
at least $\dim \cO_{x(\tau_k)}=\dim M-1$.
\end{description}

We have shown that $\sigma_0$ is a codimension $1$ saddle
singularity.  The proof of Theorem~\ref{mthm:weaknegexpFlow}
is complete.

\begin{remark}
  \label{rmk:generalsliding}
  The argument in Remark~\ref{rmk:slidehyptime} would allow
  us to replace $\tau_k$ by any $s_k\in(\tau_k,T_k)$ and so we
  would replace $\sigma_0$ by $\sigma$ in the previous
  argument, but not more, since we do not know the distance
  between $x(\tau_k)$ and $x(T_k)$.
\end{remark}
  
\subsection{The strong sectional asymptotic contracting
  case}
\label{sec:limsup-case}

We now use the previous arguments to complete the proof of
Theorem~\ref{mthm:strongnegexpFlow}. For if we assume the
stronger asymptotic contracting condition on the right hand
side of~\eqref{eq:weaknegexpflowpoint}, then we can perform all
the arguments in Subsections~\ref{sec:orbit-away-from}
and~\ref{sec:orbit-accumul-some}, and we are left to show
that the positive orbit of $x$ is not allowed to accumulate
saddle equilibrium points.

If $x\in M\setminus\sing(G)$ is such that
$\omega(x)\supsetneq\{\sigma\}$ for some
$\sigma\in\sing(G)$, then $\sigma$ must be a hyperbolic
codimension one saddle by the previous arguments, that is,
$\dim E^u_\sigma=1$.  Using the Hartman-Grobman Theorem
again, we find ourselves in a situation similar to the one
on the left hand side of Figure~\ref{fig:acumulasigma}.

More precisely, we choose 
a smooth manifold $\Sigma'$ with a cusp at $\sigma$
according to the following; see
Figure~\ref{fig:acumularsigmasup}.

\begin{lemma}\label{le:cusptransversal}
  Given a codimension one saddle singularity $\sigma$ of a
  $C^1$ vector field $G$, there exists a smooth hypersurface
  $\Sigma'$ of $M$ tangent to $E^s(\sigma)$ at $\sigma$ so
  that
  $\cos\angle(G(z),E^u_\sigma)\xrightarrow[z\to\sigma]{z\in\Sigma'}0$
  and $\Sigma'\setminus\{\sigma\}$ is a Poincar\'e section
  of the flow: that is, in a neighborhood $V$ of $\sigma$,
  for all $p\in V$, we have only one of the following
  \begin{itemize}
  \item either $\phi_tp\in V$ for all $t\ge0$ and
    $\phi_tp\xrightarrow[t\to+\infty]{}\sigma$ (i.e.
    $p\in W^s(\sigma)$);
  \item or $\phi_tp\in V$ for all $t\le0$ and
    $\phi_tp\xrightarrow[t\to-\infty]{}\sigma$ (i.e.
    $p\in W^u(\sigma)$);
  \item or $\exists t_0\in\RR: \phi_{t_0}p\in\Sigma'$ and
    $\phi_{[0,t_0]}p\subset V$ if $t_0\ge0$; or
    $\phi_{[-t_0,0]}p\subset V$ if $t_0<0$.
  \end{itemize}
\end{lemma}
We again postpone the proof of this result to
Subsection~\ref{sec:proof-pliss-lemma}.

\begin{figure}[htpb]
  \centering
  \includegraphics[height=6cm,width=6cm]{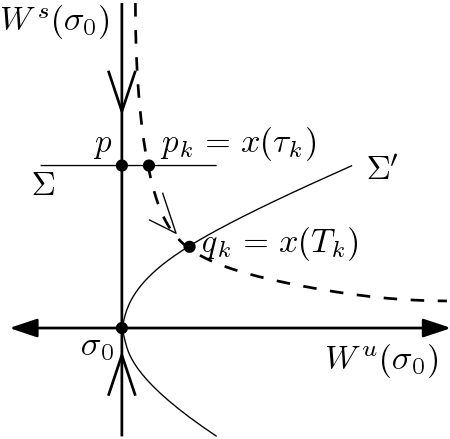}
  \caption{The strong assymptotic contracting case near a
    saddle singularity.}
  \label{fig:acumularsigmasup}
\end{figure}

From Lemma~\ref{le:cusptransversal} and since the trajectory
of $x$ satisfies the strong asymptotic contraction condition
in the right hand side of~\eqref{eq:weaknegexpflowpoint}, we
can find real valued sequences $\tau_k,T_k\nearrow\infty$
such that $q_k=x(T_k)\to\sigma$ and $q_k\in\Sigma'$; and
also $\ln\|P_x^{T_k}\|<-\xi T_k$ and $\tau_k<T_k$ is
an $e^{-\zeta/2}$-reverse hyperbolic time for $x$ with
respect to $T_n$.  We consider two cases:
\begin{description}
\item[Case A] either (perhaps for some subsequence)
  $p_k=x(\tau_k)\to p\in M\setminus\sing(G)$: in this case
  we get $p\in W^s(\sigma)$ by Lemma~\ref{le:accstablesink}
  and then conclude that $\sigma$ is a sink by
  Lemma~\ref{le:stablemanifold}.
\item[Case B] or, $p_k\to\sing(G)$.

  If $p_k\to\sigma_0\neq\sigma$, then $\sigma_0$ is again a
saddle and we use Remark~\ref{rmk:slidehyptime} to replace
$\tau_k$ by $s_k\in(\tau_k,T_k)$ so that $s_k$ is a
$e^{-\zeta/4}$-reverse hyperbolic time w.r.t. $T_n$; and
$p_k=x(s_k)\in\Sigma$ for a cross-section $\Sigma$ to $G$
through $p\in W^s(\sigma)\setminus\{\sigma\}$; see
Figure~\ref{fig:acumularsigmasup}. In addition $p_k\to p$
and we have reproduced Case A. Then $\sigma$ is a sink.

Otherwise, we have $p_k\to\sigma$ and so the segment
$x([\tau_k,T_k])$ tends to $\sigma$ when $k\nearrow\infty$.
\end{description}

We are left with $P_{p_k}^{t-\tau_k}$ a
$e^{-\zeta t/2}$-contraction for all $0<t\le T_k-\tau_k$.
Since $G(q_k)^\perp$ is very close to the expanding
direction $E^u_\sigma$, we obtain a contradiction as in the
previous Subsection~\ref{sec:conclus-proof-theore}.

More precisely, for any given fixed $t>0$ we have
$\dist\big( x([\tau_k\tau_k+t]),\sigma\big)
\xrightarrow[k\to\infty]{}0$ so we again apply
Lemma~\ref{le:Gronwall} to get
$ \|P_{x(\tau_k)}^t -\cO_{x(\tau_k+t)} e^{tDG_\sigma}\| \le
\bar\delta_k t e^{Lt} \xrightarrow[k\to\infty]{}0 $. By
the choice of $T_k$ and $\Sigma'$, we also have
$\|\cO_{x(T_k)}v\|\ge\|v\|/2$ for all $v\in E^u_\sigma$,
because $\angle(G(x(T_k)),E^u)\to\frac{\pi}2$; see
Figure~\ref{fig:acumularsigmasup}.  Then we conclude that
$E^u_\sigma=\{\vec0\}$ as in the proof of
Lemma~\ref{le:stablemanifold}.

This shows that $\sigma$ is a sink and completes the proof
of Theorem~\ref{mthm:strongnegexpFlow}.

\subsection{Weak (sectional) asymptotic expansion}
\label{sec:weak-section-expans}

Now we outline the proofs of the second part of
Theorem~\ref{mthm:negexpflow-sink} and of
Theorem~\ref{mthm:posectionalexp} since they follow very
similar lines to the asymptotic contracting case.

First we note that for the cocycles
$\tilde\Gamma(t,x)v=\ln\|D\phi_t(x)^{-1}v\|$ and
$\tilde\psi(t,x)=\ln\|(P_x^t)^{-1}v\|$ a result similar to
Lemma~\ref{le:claimDx} holds: they admit infinitesimal
generators $\tilde D_G(v)$ and $\tilde D(v)$ respectively,
which are continuous functions of $v\in T_{M\setminus\sing(G)}^1M$
and globally bounded. So the assumptions of the second part
of Theorem~\ref{mthm:negexpflow-sink} and
Theorem~\ref{mthm:posectionalexp} imply
\begin{align*}
  \liminf_{T\to\infty}
  &\frac1T\sup_{v\in T_x^1M}\int_0^T\tilde
  D_G(\Phi_sv)\,ds
  \qand
  \\
  \liminf_{T\to\infty}\frac1T
  \sup_{v\in T_x^1M\cap G^\perp}
  &\int_0^T\tilde D(\wh{\Phi_s}x)\,ds\;
  \text{ or }\;
    \limsup_{T\to\infty}
    \sup_{v\in T_x^1M\cap G^\perp}\frac1T\int_0^T\tilde D(\wh{\Phi_s}x)\,ds,
\end{align*}
are negative.
We then apply the same arguments as in the proof of
Theorems~\ref{mthm:negexpflow-sink},~\ref{mthm:weaknegexpFlow}
and~\ref{mthm:strongnegexpFlow} because the analogous to
Lemma~\ref{le:cballs} and
Proposition~\ref{pr:tubecontraction} for the (sectional)
expanding case at hyperbolic times are also true. We state
the results below and explain what we mean by hyperbolic times
in this setting.

In any of the $\liminf$ assumptions above, we can find
$\zeta>0$ and $T_n\nearrow\infty$ so that
\begin{align*}
  \sup_{v\in T_x^1M}\int_0^{T_n}
  \tilde D_G(\Phi_sv)\,ds\le-\frac{\zeta T_n}2\;
  \text{ or }\;
  \sup_{v\in T_x^1M\cap G^\perp}\int_0^{T_n}
  \tilde D(\wh{\Phi_s}x)\,ds\le-\frac{\zeta T_n}2
\end{align*}
and apply Theorem~\ref{thm:plissflows} to the functions
$\tilde H_G(t,v)=\int_{T_n-t}^{T_n} \tilde D_G(\Phi_sv)\,ds$
or $\tilde H(t,v)=\int_{T_n-t}^{T_n} \tilde D(\wh{\Phi_s}v)\,ds$,
respectively. We obtain $\tau_n<T_n$ with
$\tau_n\nearrow\infty$ and $T_n-\tau_n\nearrow\infty$ such
that for $0<t<\tau_n$
\begin{align}\label{eq:hyptime1flow}
  \ln\|\big(D\phi_{\tau_n-t}(\phi_tx)\big)^{-1}\|
  &=
    \sup_{v\in T_x^1M}\int_t^{\tau_n}\tilde  D_G(\Phi_sv)\,ds
  \le
    -\frac{\zeta}4(\tau_n-t)\;\text{ or}
  \\\label{eq:hyptimeLPF}
  \ln\|(P^{\tau_n-t}_{\phi_tx})^{-1}\|
  &=
    \sup_{v\in T_x^1M\cap G^\perp}\int_t^{\tau_n}\tilde  D_G(\wh{\Phi_s}v)\,ds
  \le
    -\frac{\zeta}4(\tau_n-t),
\end{align}
respectively. These are reverse hyperbolic times, where we
have uniform infinitesimal contractions from
$\phi_{\tau_n}x$ to $\phi_t(x)$ for $0<t<\tau_n$ with
respect to the flow in~\eqref{eq:hyptime1flow} or the Linear
Poincar\'e Flow in~\eqref{eq:hyptimeLPF}.

\begin{proof}[Proof of the second part of
  Theorem~\ref{mthm:negexpflow-sink}]
  In the case \eqref{eq:hyptime1flow}, $[\tau_n]$ is a
  hyperbolic time for $y_n=\phi_{\tau_n-[\tau_n]}(x)$ with
  respect to the $C^1$ diffeomorphism $f=\phi_1$. So we can
  apply Lemma~\ref{le:backcontr} to get infinitely many
  backward contracting balls to which we can apply the
  nested contractions argument from
  Subsection~\ref{sec:nested-contract-argu}. We obtain a
  repelling periodic orbit $p$ for $f$ such that
  $p\in B(x(\tau_n),\delta_1)$; thus also a repelling
  periodic orbit $\cO_G(p)$ for $G$. But
  $G(p)=D\phi_k(p)\cdot G(p)=Df^k(p)\cdot G(p)$ for some
  $k\ge1$ and so $1$ would be an eigenvalue of $Df^k(p)$ if
  $G(p)\neq\vec0$, contradicting the expansion of the
  derivative map at repelling periodic points. This shows
  that $G(p)=\vec0$, hence $p=\sigma\in\sing(G)$ is a
  repelling equilibrium (a source).

  Moreover, by the properties of backward contracting balls,
  we get $\dist(y_n,\sigma)\le e^{-\zeta\tau_n/4}$ where we
  can take $n$ larger than any predetermined quantity. Hence
  the distance between $\phi_{[0,1]}(x)$ and $\cO_G(p)$ is
  zero and $x$ is a source. This completes the proof of
  Theorem~\ref{mthm:negexpflow-sink}.
\end{proof}

\begin{proof}[Proof of Theorem~\ref{mthm:posectionalexp}]
  In the case~\eqref{eq:hyptimeLPF} we use the following,
  whose proof is left to the reader.

  \begin{proposition}[Existence of backward contracting
    balls]
    \label{pr:tubexpansion}
    Let $\tau_n<T_n$ be the pair of strictly increasing
    sequences obtained above. For every $\delta_0>0$ there
    exists $\xi_0>0$ satisfying \eqref{eq:unifcontquotient}
    such that, if
    $\dist(\phi_tx,\sing(G))\ge d_0, \forall
    t\in[0,\tau_n]$, then for each $s\in(0,\tau_n]$ there
    exists a $C^1$ smooth well-defined diffeomorphism with
    its image $R_s:S_{x(\tau_n)}(\xi_0)\to S_{x(s)}(\xi_0)$
    such that $R_s$ is a Poincar\'e map (for the
    time-reversed flow), $R_s(x(\tau_n))=x(s)$ and $R_s$ is
    an $e^{-\frac{\zeta}4(\tau_n-s)}$-contraction.
  \end{proposition}

  We have now all the tools to apply the same arguments in
  Subsections~\ref{sec:orbit-away-from},~\ref{sec:orbit-accumul-some}
  and~\ref{sec:limsup-case} to conclude the proof of
  Theorem~\ref{mthm:posectionalexp}.
\end{proof}


\subsection{Proofs of Lemmata}
\label{sec:proof-pliss-lemma}

Now we present the proofs of the technical result previously used as
tools in this section.

\begin{proof}[Proof of Lemma~\ref{le:claimDx}]
  We prove items (1-3) for $D_\pm(x)$ only since for $D_{G\pm}(x)$
  the arguments are completely analogous, but much simpler,
  and are left to the reader.
  
  To prove the continuity of
  $x\in M\setminus\sing(G)\mapsto D_\pm(x)$ note that
  $\|P^h_y\|\xrightarrow[h\to0]{}1$ and
  \begin{align}\label{eq:diffpsi}
    \frac1h\ln\frac{\|P^h_x\|}{\|P^h_y\|}
    &=
      \frac1{2h}\ln\frac{\|P^h_x\|^2}{\|P^h_y\|^2}
      =
      \frac1{2h}\ln\left(1+\frac{\|P^h_x\|-\|P^h_y\|^2}{\|P^h_y\|^2}\right)
      \le
      \frac{\|P^h_x\|^2-\|P^h_y\|^2}{2h\|P^h_y\|^2}.
  \end{align}
  We express the time derivative of
  $\|P_x^h\|^2=\sup_{\|u\|=1}\langle P_x^hu,P_x^hu\rangle$
  as follows. On the one hand, writing $\hat G=G/\|G\|$
  and $\phi_hz=z_h$ for any $z\in M, h\in\RR$, we get
  \begin{align*}
    (P_x^h)'
    &=
      (\cO_{\phi_hx}D\phi_x(x))'
    =
    \big(
    D\phi_x(x) - \langle D\phi_x(x), \hat
    G(x_h)\rangle \hat  G(x_h)
      \big)'
    \\
    &=
      DG_{x_h}D\phi_h(x)-\langle DG_{x_h}D\phi_h(x),\hat
      G(x_h)\rangle \hat G(x_h)
    \\
    &\quad
      -\langle D\phi_h(x),\hat G(x_h)'\rangle \hat G(x_h)
      -\langle D\phi_h(x),\hat G(x_h)\rangle \hat G(x_h)'
  \end{align*}
  and since $\hat G(x_h)'= \cO_{x_h}DG_{x_h}\hat G(x_h)$ we
  finally obtain
$
  \langle P_x^h,P_x^h\rangle'=2\langle \cO_{x_h}DG_{x_h}P_x^h ,P_x^h\rangle.
$

Along this proof, we are implicitly assuming that
$x_h=\phi_hx$ is in the range of $\exp_x$, identifying the
tangent spaces $T_xM$ and $T_{x_h}M$ through
$D(\exp_x)_{\exp_x^{-1}x_h}$ and writing $DG_y v$ for
$\nabla_vG(y)$ with $y\in M, v\in T_yM$, where $\nabla$ is
the Levi-Civita connection associated to the Riemannian
metric of $M$.

On the one hand, for a singular vector $u_h\in G(x)^\perp$
corresponding to the largest singular value\footnote{The largest
  coefficient of the orthogonal diagonalization of the quadratic form
  $v\mapsto\langle (P_x^h)^*P_x^h v,v\rangle$, where $*$ denotes the
  adjoint operator with respect to the inner product.} of $P_x^h$ and
$|h|$ sufficiently small, we have for some intermediate value $s=s(h)$
so that $0<|s(h)|<|h|$
\begin{align*}
  \|P_x^h\|^2
  =
  1+2\int_0^h\langle \cO_{x_s}DG_{x_s}P_x^su_h
  ,P_x^su_h\rangle\,ds
  =
  1+2h\langle \cO_{x(s)}DG_{x(s)}P_x^{s(h)}u_h
  ,P_x^{s(h)}u_h\rangle.
\end{align*}
Therefore $P^{s(h)}_xu_h$ is an eigenvector associated to the largest
eigenvalue of\footnote{ This is given by
  $\sup\{\Re\lambda: \lambda\in\spec(\cO_{x(s)}DG_{x(s)})\}$ which is
  the largest real coefficient in the orthogonal diagonalization of
  the quadratic form
  $\langle \cO_{x(s)}DG_{x(s)}v,v\rangle=\sum_ia_iX_i^2$ where
  $v=\sum_iX_ie_i$ for some orthonormal basis $e_1,\dots,e_d$ of
  $T_xM$.}  $(\cO_{x(s)}DG_{x(s)}+(\cO_{x(s)}DG_{x(s)})^*)/2$.  On the
other hand, for $v_h\in G(y)^\perp$ corresponding to the largest
singular value of $P_y^h$ we can find some intermediate value
$\bar s=\bar s(h)$ so that $0<|\bar s(h)|<|h|$
\begin{align*}
  \eqref{eq:diffpsi}
  &=
    \frac{
    \langle \cO_{x_s}DG_{x_s}P_x^{s(h)}u_h ,P_x^{s(h)}u_h\rangle
    -
    \langle \cO_{\bar y_s}DG_{\bar y_s}P_y^{\bar s(h)}v_h
    ,P_y^{\bar s(h)}v_h\rangle}
    {\|P^h_y\|^2}.
\end{align*}
Hence, when $h\to0$, using the compactness of the unit
sphere, we get $u,v$ accumulation unit vectors of the
families $(u_h),(v_h)$ so that
\begin{align}\label{eq:diffSVD}
  \lim_{\delta\searrow0}\sup_{0<|h|<\delta}
  \left|\frac1h\ln\frac{\|P^h_x\|}{\|P^h_y\|}\right|
  \le
    \left|\langle \cO_{x}DG_{x}u ,u\rangle
    -
    \langle \cO_{y}DG_{y}v ,v\rangle\right|
\end{align}
and both $u,v$ are eigenvectors associated to the
largest eigenvalues of the operators
$(\cO_{x}DG_{x} +(\cO_{x}DG_{x})^*)/2$  and
$(\cO_{y}DG_{y} +(\cO_{y}DG_{y})^*)/2$, respectively.
Since these are  symmetric operators, we apply the following.

\begin{lemma}
  {\cite[Theorem III.6.11]{Kato95}} Let $A,B$ be selfadjoint
  operators of $\RR^d$ and $C=A-B$, whose eigenvalues
  repeated with multiplicities we write as
  $\alpha_1\le\dots\le\alpha_d$, $\beta_1\le\dots\le\beta_d$
  and $\gamma_1\le\dots\le\gamma_d$ respectively. Then
  $\sum_{i=1}^d|\alpha_i-\beta_i| \le\sum_{i=1}^d|\gamma_i|$.
\end{lemma}
As a direct consequence of this result, recall the standard
bound for the spectral radius
\begin{align*}
r(C)=\sup\{|\lambda|:
\lambda\in\spec(C)\}\le\|C\|=\|A-B\|
\end{align*}
and so $|\alpha_i-\beta_i|\le d r(C)\le d\|A-B\|$ for each
$1\le i\le d$.

Going back to \eqref{eq:diffSVD} writing
$A=(\cO_{x}DG_{x} +(\cO_{x}DG_{x})^*)/2$ and
$B=(\cO_{y}DG_{y} +(\cO_{y}DG_{y})^*)/2$ we get\footnote{Recall that
  we are implicitly assuming that $y$ is in the range of $\exp_x$.}
$~\eqref{eq:diffSVD} \le \dim
M\cdot\|\cO_{x}DG_{x}-\cO_{y}DG_{y}\|$. This together with
\eqref{eq:diffpsi} implies
\begin{align*}
  D_+(x)
  &=
  \lim_{\delta\searrow0}\sup_{0<h<\delta}\frac{\ln\|P^h_x\|}h
  \le
  \lim_{\delta\searrow0}\left(
  \sup_{0<h<\delta}\frac{\ln\|P^h_y\|}h
  +
  \sup_{0<h<\delta}
  \left|\frac1h\ln\frac{\|P^h_x\|}{\|P^h_y\|}\right|
  \right)
  \\
  &=
  D_+(y)+\dim M\cdot\|\cO_{x}DG_{x}-\cO_{y}DG_{y}\|)
\end{align*}
and since we can exchange $x$ and $y$ this completes the proof of the
continuity of $D_+(x)$. Moreove, we can argue with $h\nearrow0$ using
the same inequality~\eqref{eq:diffSVD}, thus obtaining continuity for
$D_-(x)$ also . This completes the proof of item (2).

For item (1), 
we note that
$\|P^h_xv\|\ge\|(P^h_x)^{-1}\|^{-1}\|v\|, v\in G(x)^\perp$
and that
  \begin{align*}
    (P^h_x)^{-1}=\cO_x\circ D\phi_h(x)^{-1}\circ\cO_{\phi_h x}
    =\cO_x\circ D\phi_{-h}(\phi_hx)\circ\cO_{\phi_h x}
  \end{align*}
  so
  $\|(P^h_x)^{-1}\|\le\|D\phi_{-h}(\phi_hx)\|\le e^{|h|L}$
  from Gronwall's Inequality and consequently
  $\ln\|P^h_x\|^{1/h}\ge-L$ for all
  $h\in\RR, x\in M\setminus\sing(G)$.

  Hence $D_\pm(x)\ge -L$ and analogously $D_\pm(x)\le L$ for all
  $x\in M\setminus\sing(G)$. 

  For item (3): from the continuity of $x\mapsto D(x)$ and
  Lemma~\ref{le:subdiff} we deduce the
  relation~\eqref{eq:lnLPFint}, since $t\mapsto\psi_tx$
  satisfies
  \begin{align*}
    \limsup_{h\to0+}\frac{\psi_{t+h}x-\psi_tx}h
    &\le
    \limsup_{h\to0+}\frac{\psi_h(\phi_tx)+\psi_tx-\psi_tx}h
    =
    D_+(x)\qand
    \\
    \liminf_{h\to0-}\frac{\psi_{t+h}x-\psi_tx}h
    &\ge
      \liminf_{h\to0-}\frac{\psi_h(\phi_tx)+\psi_tx-\psi_tx}h
      =D_-(x).
  \end{align*}
  Hence, for any partition $0=t_0<t_1<\dots<t_k=T$ of the
  interval $[0,T]$ with width
  $\delta=\sup_{1\le i\le k}(t_{i+1}-t_i)$ we get on the one
  hand
\begin{align*}
  \psi_Tx
  &=
    \sum_{i=1}^k\frac{\psi_{t_{i+1}}x-\psi_{t_i}x}{t_{i+1}-t_i}(t_{i+1}-t_i)
    \le
    \sum_{i=1}^k
    \sup_{0<h<\delta}\left(\frac{\psi_h(\phi_{t_i}x)}h\right) (t_{i+1}-t_i)
\end{align*}
and since
$\lim_{\delta\searrow0}
\sup_{0<h<\delta}\left(\frac{\psi_h(\phi_{t_i}x)}h\right)
=D(\phi_{t_i}x)$ we obtain
$ \psi_Tx\le\int_0^TD_+(\phi_sx)\,ds$. On the other hand
\begin{align*}
  \psi_Tx
  &=
    \sum_{i=1}^k\frac{\psi_{t_{i}}x-\psi_{t_{i+1}}x}{t_{i}-t_{i+1}}(t_{i+1}-t_i)
    \ge
    \sum_{i=1}^k
    \inf_{-\delta<h<0}\left(\frac{\psi_h(\phi_{t_{i+1}}x)}h\right) (t_{i+1}-t_i)
\end{align*}
and since
$\lim_{\delta\searrow0}
\inf_{-\delta<h<0}\left(\frac{\psi_h(\phi_{t_{i+1}}x)}h\right)
=D(\phi_{t_{i+1}}x)$ we also get $\psi_Tx\ge\int_0^T
D_-(\phi_sx)\,ds$ and obtain~\eqref{eq:lnLPFint}.
  This completes the proof of the lemma.
\end{proof}  


\begin{proof}[Proof of Lemma~\ref{le:equivstat}]
  Let us assume that
  $-\xi=\liminf_{T\to+\infty}\ln\|D\phi_T(x)\|^{1/T}<0$. Then for each
  $\epsilon>0$ there exists a sequence $T_n\nearrow\infty$ so that for
  all $v\in T^1_xM$ we have from Lemma~\ref{le:subdiff}
  \begin{align*}
    H_G(T_n,v)=\int_0^{T_n}D_G(\Phi_sv)\,ds
    =
    \ln\|D\phi_{T_n}(x)\cdot v\|\le \ln\|D\phi_{T_n}(x)\|\le-\xi T_n.
  \end{align*}
  This proves the first statement of the lemma. The proof of the other
  statements is similar and left to the reader.
\end{proof}


\begin{proof}[Proof of Lemma~\ref{le:Gronwall}]
  We can assume without loss of generality that
  $\phi_{[0,t]}q_n$ is in the range of $\exp_\sigma$ for all
  $n\ge1$ so that we can identify all tangent spaces.  Then
  we note that
  \begin{align}\label{eq:Gronwall1}
    \|P_{q_n}^h-\cO_{\phi_hq_n}D\phi_h\sigma\|
    =
    \|\cO_{\phi_hq_n}D\phi_hq_n-\cO_{\phi_hq_n} e^{h
    DG_\sigma}\|
    \le
    \|D\phi_hq_n-e^{h DG_\sigma}\|
  \end{align}
  so we need only estimate the last norm. For that we use
  Gronwall's Inequality as follows: $D\phi_h(z)$ is the
  solution of the Linear Variational Equation
  $\dot Z=DG_{\phi_h z}\cdot Z$ with $Z(0)=Id$, $z=q_n$ or
  $\sigma$, for $h\in[-\epsilon,\epsilon]$ and $\epsilon>0$
  small, in the coordinates of a local chart of $M$
  containing both $q_n$ and $\sigma$. Then we can write
\begin{align*}
    D\phi_h(q_n)-e^{h DG_\sigma}
    &=
      \int_0^h\big(
      DG_{\phi_s q_n}\cdot D\phi_s(q_n)-
      DG_{\sigma}\cdot e^{sDG_\sigma}\big)\,ds
    \\
    &=
      \int_0^h DG_{\phi_sq_n}\cdot\big(D\phi_s(q_n)-
      e^{sDG_\sigma}\big)\,ds
      +
      \int_0^h\big(DG_{\phi_sq_n}-DG_{\sigma}\big)
      \cdot e^{sDG_\sigma}\,ds
  \end{align*}
  and taking norms we obtain
  $ \beta(h)\le\alpha(h)+\int_0^h\gamma(s)\beta(s)\,ds$
  where we set $\gamma(s)=\|DG_{\phi_sq_n}\|$;
  $ \alpha(h)=
  \int_0^h\left\|\big(DG_{\phi_sq_n}-DG_{\sigma}\big)\right\|
  \cdot \|e^{sDG_\sigma}\|\,ds$ and
$  \beta(h)=\|D\phi_h(q_n)-e^{h DG_\sigma}\|$. We
  conclude\footnote{See e.g.\cite[Lemma 4.7]{PM82} and
    \cite[Theorem 2.1]{chicone06}.}
  $ \beta(h)\le\alpha(h)\exp\int_0^h\gamma(s)\,ds$.

  Now we
  have $\gamma(h)\le L$ and $\|e^{h DG_\sigma}\|\le e^{hL}$,
  so we arrive at
  \begin{align*}
    \alpha(h)\le he^{hL}\sup_{0\le s\le h}
    \left\|DG_{\phi_sq_n}-DG_\sigma\right\|
    =
    \bar\delta_n he^{hL}.
  \end{align*}
  ensuring that we can bound~\eqref{eq:Gronwall1} by
  $\bar\delta_n hL e^{hL}$.

  Finally, if we know that the trajectory $\phi_{[0,t]}q_n$
  is close to $\sigma$, then we can perform the above
  integrations and estimations for $h=t$ and complete the
  proof of the lemma.
  \end{proof}

We now prove a second technical lemma.

\begin{proof}[Proof of Lemma~\ref{le:cusptransversal}]
  We can obtain $\Sigma'$ simply writting
  $T_\sigma M=E^s(\sigma)\oplus E^u(\sigma)\cong
  \RR^{\dim M-1}\times\RR$ and setting
  $\Sigma'=\exp_\sigma\Sigma_0$ where
  $\Sigma_0=\{(u,v)\in E^s\times E^u: v=\|u\|^2\}$. Indeed,
  the linear vector field
$
  w=(w_s,w_u)\in T_\sigma M \mapsto DG_\sigma w=(Aw_s,\xi
  w_u),
$
  for fixed $\xi>1$ and $A\in GL(E^s(\sigma))$ with
  $\Im\spec(A^s)\subset\RR^-$, over $z\in\Sigma_0$ has angle
  with the vertical direction $(0,1)$ which tends to zero
  when $z\to0$. In fact,
  \begin{align*}
    \cos\angle((Az_s,\xi z_u)),(0,1))=\frac{\xi
    z_u}{\|(Az_s,\xi z_u)\|}
    \xrightarrow[z\to0]{z\in\Sigma'}0
    \iff
    \frac{|z_u|}{\|z_s\|}\xrightarrow[z\to0]{z\in\Sigma'}0
  \end{align*}
  is a consequence of $z_u=\|z_s\|^2$ (that is, $z\in\Sigma'$)
  together with $z\to0$. Since the vector $(0,1)$ is the direction of
  $E^u_\sigma$, we have
    \begin{align}\label{eq:constangle}
    \cos\angle(G(z),E^u_\sigma)\xrightarrow[z\to\sigma]{z\in\Sigma'}0
  \end{align}
  in the linearized case.  Hence, $G$ and $\Sigma'$ satisfy
  \eqref{eq:constangle} in a small enough neighborhood of
  $\sigma$, because the vector field
  $\tilde G=D(\exp_\sigma)^{-1}G$ on a neighborhood of $0$
  in $T_\sigma M$ satisfies for each $v\in T_\sigma M$
  \begin{align*}
    D\tilde G_0 v
    &=
      \lim_{t\to0}\frac{D(\exp_\sigma)^{-1}G(\exp_\sigma(tv))}t
      =
      D(\exp_\sigma)^{-1}\lim_{t\to0}\frac{G(\exp_\sigma(tv))}t
      =
      DG_\sigma v.
  \end{align*}
  Consequently
  $\|DG_\sigma w - D(\exp_\sigma)^{-1}_wG(\exp_\sigma
  w)\|/\|w\|\xrightarrow[w\to0]{}0$ and thus
  \begin{align*}
   \left\langle\frac{G(\exp_\sigma w)}{\|G(\exp_\sigma w)\|},
    D(\exp_\sigma)_0\cdot(0,1)\right\rangle
    =
    \left\langle \frac{D(\exp_\sigma)_w\tilde
    G(w)}{\|D(\exp_\sigma)_w\tilde G(w) \|},
    D(\exp_\sigma)_0\cdot(0,1)\right\rangle
  \end{align*}
  has the same limit as $w\to0$ along $\Sigma_0$ as \eqref{eq:constangle},
  since $D(\exp_\sigma)_0=Id$ and $\Sigma'=\exp_\sigma \Sigma_0$.

  For the sectional property of $\Sigma$, by the
  Hartman-Grobman Linearization Theorem, the flow of
  $DG_\sigma$ in a neighborhood $V$ of $0$ in $T_\sigma M$
  is topologically conjugated to the flow of $G$ is a
  neighborhood $U$ of $\sigma$ in $M$: for any given
  $\delta>0$ we can find a homeomorphism $h:V\to U$ such
  that $\|Id-h^{-1}\exp_\sigma\|<\delta$ and
  $\phi_th(w)=h(e^{t DG_\sigma}w)$ for $w\in T_\sigma M$
  satisfying $e^{sDG_\sigma}w\in V,\forall0\le s\le t$.

  It is thus enough to prove that
  $\tilde\Sigma=h^{-1}(\Sigma\setminus\{\sigma\}
  =h^{-1}\exp_\sigma(\Sigma_0\setminus\{0\})$ is a
  Poincar\'e section of the linearized flow. Since
  $h^{-1}\exp_\sigma$ is close to the identity map, without
  loss of generality we can assume that coordinates have
  been chosen on $T_\sigma M$ so that
  \begin{itemize}
  \item $\tilde\Sigma$ is a graph of a Lipschitz function
    $g:E^s\cap B(0,2)\to E^u\cong\RR$ satisfying $\lip(g)<1$,
    $g(u)\ge0$ and $g(u)=0\implies
    u=0$. 
    We set $a=\inf\{ g(u): u\in E^s, \|u\|=1\}$.
  \item for $z=(z_s,z_u)$ with $\|z_s\|=1$ and
    $0<z_u<g(z_s), t>0$, from
    $e^{t DG_\sigma}z=(e^{At}z_s,e^{\xi t}z_u)$ we deduce
  \begin{align*}
    a=e^{\xi t}z_u\iff t=\ln (a/z_u)^{1/\xi} \qand
    \|e^{At}z_s\|\le e^{-\lambda t}=(z_u/a)^{\lambda/\xi}
    \xrightarrow[z_u\to0]{}0,
  \end{align*}
  where $\lambda>0$ is such that $-\lambda\ge\Im(\alpha),
  \forall\alpha\in\spec(A)$.
  \end{itemize}
  Thus the function
  $F(x,y)=g(x)-y, (x,y)\in \RR^{\dim M-1}\times \RR$ is such
  that $F(z_s,z_u)=g(z_s)-z_u>0$ and
  $F(e^{tDG_\sigma}z)=F(e^{At}z_s)-a<0$ for all $z_u$
  sufficiently close to $0$, showing that there exists
  $s=s(z_s,z_u)\in(0,t)$ such that $F(e^{sDG_\sigma}z)=0$,
  that is, $e^{sDG_\sigma}z\in\tilde\Sigma$.
  
  Moreover, if $F(e^{\bar s DG_\sigma} z)=0$ for some
  $\bar s>s$, then
  \begin{align*}
    z_u e^{\xi s}(1-e^{\xi(\bar s-s)})
    &=
    g(e^{A\bar s}z_s)-g(e^{As}z_s)
    \le
    \lip(g)\|e^{As}(e^{A(\bar s -s)}-I)z_s\|
  \end{align*}
which implies for some $0<\zeta<\bar s-s$ by the Mean Value Inequality
\begin{align*}
  z_u
  &\le
    \lip(g)\frac{e^{-(\lambda+\xi)s}}{1-e^{\xi(\bar s-s)}}
    \|e^{A(\bar s -s)}-I\|
    \le
    \lip(g)\frac{(z_u/a)^{1+\lambda/\xi}}{1-e^{\xi(\bar s-s)}}
    \|A e^{A\zeta}(\bar s-s)\|
\end{align*}
and so we arrive at
\begin{align*}
  0<a
  \le
  \lip(g)\left(\frac{z_u}a\right)^{\lambda/\xi}
  \frac{\xi(\bar s-s)}{1-e^{\xi(\bar
  s-s)}}\frac{\|A\|}{\xi}e^{-\lambda(\bar s -s)}
  \le
  Const\cdot z_u^{\lambda/\xi}
\end{align*}
yielding a contradiction for all small enough $z_u>0$. We
conclude that there exists a unique $s=s(z)$ so that the
future trajectory of $z$ under the flow of $G$ intersects
$\Sigma'$ for all $z$ in a small enough neighborhood of
$\sigma$.

  Hence $\Sigma'\setminus\{\sigma\}$ is a Poincar\'e section
  for the flow $G$, completing the proof of the lemma.
\end{proof}

Finally, we prove the Lemma of Pliss for flows.

\begin{proof}[Proof of Theorem~\ref{thm:plissflows}]
  Observe first that we can assume without loss of
  generality that $H$ is of class $C^2$. Indeed, let $H$ be
  differentiable satisfying $H (0) = 0, H (T ) < cT$ and
  $\inf(H') > A$. 

  If the statement of the theorem is true for any
$\tilde H:[0,T]\to\RR$ of class $C^2$, then we choose such
$\tilde H$ so that $\tilde H(0)=0$ and
\begin{align*}
\sup_{t\in[0,T]}\{|H(t)-\tilde H(t)|,|H'(t)-\tilde
  H'(t)|\}<\tilde\epsilon
\end{align*}
for some small $0<\tilde\epsilon<\epsilon$. We obtain
\begin{align*}
  \tilde H(T)&=(\tilde H-H)(T)+H(T)<\tilde\epsilon+cT\qand
  \\
  \inf \tilde H'
  &=\inf(H'-(\tilde H'-H'))\ge A-\tilde\epsilon,
\end{align*}
and writting for $\delta>0$
\begin{align*}
  \tilde \cH_{\delta}=\{\tau \in [0, T ] : \tilde H (s)
  -\tilde H (\tau ) < (c +\delta)(s - \tau ), \,
  \text{for all}\, s \in [\tau, T ]\}.
\end{align*}
then we get $|\tilde\cH_\epsilon|\ge T\tilde\theta$ with
\begin{align*}
  \tilde\theta
  =\frac{\epsilon}{c+\tilde\epsilon/T+\epsilon-(A-\tilde\epsilon)}
  =\frac{\epsilon}{c+\epsilon-A+\tilde\epsilon(1+1/T)}
\end{align*}
and also
$\tilde\cH_\epsilon\subset\cH_{\epsilon+\tilde\epsilon(1+1/T)}$.

Therefore, since $\cH_\epsilon=\cup_{n\ge1}\cH_{\epsilon+1/n}$,
we conclude that for each small enough $\tilde\epsilon>0$ we
get $\tilde H$ of class $C^2$ which is
$\tilde\epsilon$-$C^1$-close to $H$ and
$|\cH_\epsilon|\ge|\tilde \cH_\epsilon|\ge\tilde\theta
T$. Letting $\tilde\epsilon\to0$ we obtain
$\tilde\theta\to \theta=\epsilon/(c+\epsilon-A)$ as we need.

Let now $\epsilon > 0, A, c$ and $H$ be as in the statement
of the theorem with $H$ of class $C^2$, and define
$G(s) = H(s) - (c + \epsilon)s$.  Since we have already
shown that approximating $H$ in the $C^1$ topology does not
change the conclusions of the statement of the theorem, we
may also assume without loss of generality that $G$ does not
have degenerate critical points; that is, $G'(x) = 0$ if and
only if $G''(x) \neq 0$; and, moreover, that its critical
values are all distinct. This can be done replacing $H$ by a
$C^2$-close Morse function in what follows.

\begin{figure}[htpb]
  \centering
  \includegraphics[width=11cm]{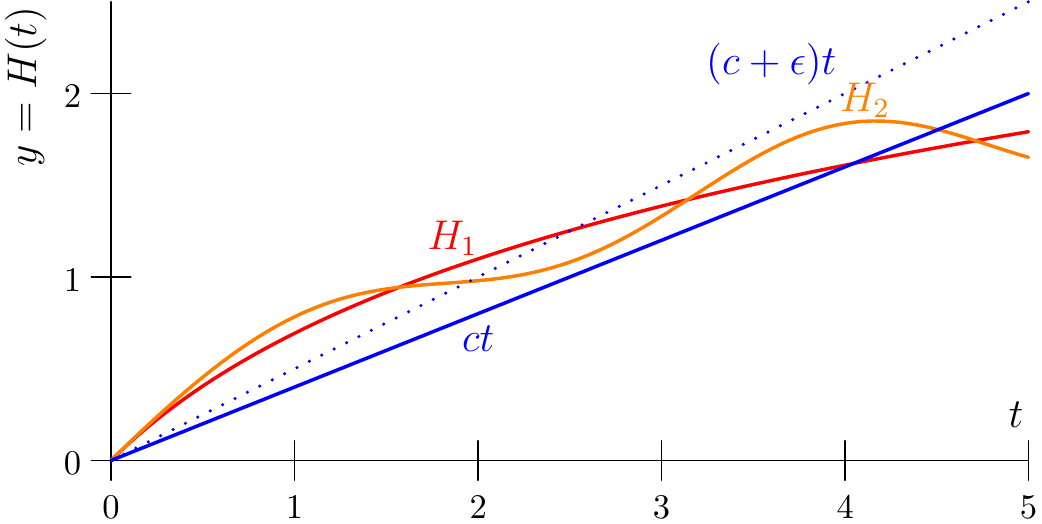}
  \includegraphics[width=11cm]{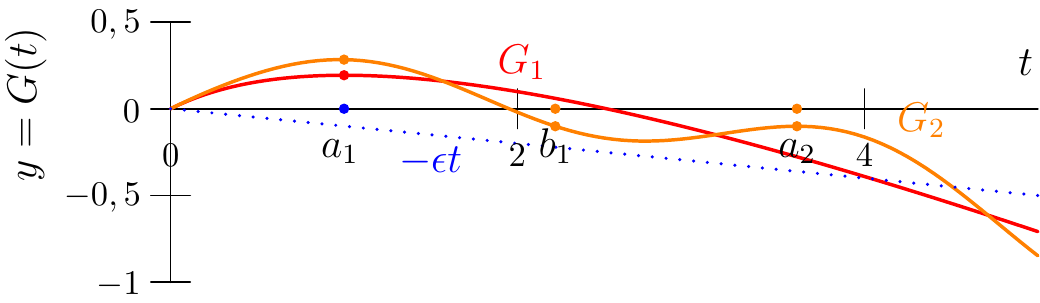}
  \caption{\label{fig:Plissfluxo} Illustrative example with the graphs
    of $H_1(t)=\log(t+1)$, $H_2(t)=(1+\sin(2t)/7)\cdot H_1(t)$,
    $c=1/2$ and $\epsilon=1/10$ above; and also
    $G_i(t)=H_i(t)-(c+\epsilon)t, i=1,2$ and the points $a_i,b_i$
    below. For $G_1$ we only have $a_1=1$; but for $G_2$ we have
    $a_1<b_1<a_2$.}
\end{figure}

Now $G(0) = 0$ and $G(T ) < -\epsilon T$, so it is possible to define
two (perhaps finite) increasing sequences, say $(a_i)_{i=1}^n$
consisting of critical points of $G$ such that $G(x) < G(a_i)$
\emph{for every} $x > a_i$ (if this sequence is finite, we set the
last point $a_{n} = T$; otherwise $a_i \to T$) and $(b_i)_{i=1}^n$ as
the smallest $b > a_i$ such that $G(b_i ) = G(a_{i+1})$; see
Figure~\ref{fig:Plissfluxo}.

More precisely, we define the sequence $(a_i)$ and $(b_i)$
recursively: $a_1=0$ if $G(t)<0$ for $t>0$; otherwise
$a_1=\inf\{s>0: G'(s)=0 \qand G(s)>G(t), \forall t>s\}$; now
inductively for $i\ge1$
\begin{align*}
  a_{i+1}&=\inf\{s>a_i: G'(s)=0 \qand G(s)>G(t), \forall
           t>s\} \qand
  \\
  b_i&=\inf\{s>a_i: G(s)=G(a_{i+1})\}.
\end{align*}
Clearly $b_i \le a_{i+1}$, and 
\begin{itemize}
\item $a_1$ is a global maximum of $G$ and $G(a_1)\ge0$;
\item each $a_i$ is a local maximum of $G$;
\item $G'(t)\neq0$ for $a_i<t\le b_i$,
  otherwise there would be a critical point
  $\xi<b_i<a_{i+1}$ with the properties of $a_{i+1}$,
  contradicting the inductive definition. In addition,
\item 
  $G(t)\le G(a_{i+1})$ for $b_i<t<a_{i+1}$ for otherwise
  there would be a critical point $\xi\in(b_i,a_{i+1})$ with
  the properties of $a_{i+1}$, again contradicting the
  inductive definition.
\end{itemize}
Letting $B=-\inf G'$, then the Mean Value Theorem ensures that
\begin{align*}
\frac{G(a_i ) - G(b_i)}B \le b_i - a_i
\end{align*}
and we claim that the union $\cup_i (a_i , b_i )$ is
contained in $\cH_\epsilon$. Indeed
\begin{align*}
  H (s) - H (\tau )
  &=
    H (s)-(c+\epsilon)s-[H (\tau)-(c +\epsilon)\tau]
    + (c + \epsilon)(s - \tau )
  \\
  &= G(s) - G(\tau ) + (c + \epsilon)(s - \tau ).
\end{align*}
and so $  H (s) - H (\tau ) < (c +
\epsilon)(s - \tau )$ if, and only if, $G(s) < G(\tau)$.

Now we let $\tau \in (a_i , b_i )$ for some $i$ and argue by
contradiction: let us assume that for a given $t > \tau$ we
have $G(t) \ge G(\tau )$. Since there are no critical points
in $(\tau,b_i]$, we must have $t\ge b_i$. But this is
impossible, because $G(t)\le G(a_{i+1})=G(b_i)<G(\tau)$ for
$b_i<t<a_{i+1}$ and $G(t)\le G(a_{i+1})$ for all $t\ge
a_{i+1}$ by construction.
This contradiction shows that $\tau\in\cH_\epsilon$, as claimed.
Therefore
\begin{align*}
|\cH_\epsilon|
&\ge
\sum_{i=1}^n
(b_i - a_i)
\ge
\frac1B \sum_{i=1}^n [G (a_i ) - G(b_i )]
  \\
  &=\frac1B\sum_{i=1}^n
[G(a_i ) -G(a_{i+1} )] =
\frac1B[G(a_1 ) - G(T )].
\end{align*}
Since $G(a_1)\ge 0$ we obtain
$
  |\cH_\epsilon| \ge -\frac{G(T)}B \ge \frac{\epsilon T}B.
$
Notice that since $G'(t) = H'(t) - (c + \epsilon)$, to get a
non-trivial result we need $B>0$, that is,
$A<\inf H'<c+\epsilon$.  Finally
\begin{align*}
  B=
  -\inf  G'
  =
  \sup (c+\epsilon-H')
  =
  c+\epsilon-\inf H'
  \le
  c+\epsilon-A
\end{align*}
and so
$
  |\cH_\epsilon|\ge\frac{\epsilon}B
  T\ge\frac{\epsilon}{c+\epsilon-A} T
$
completing the proof of the theorem, by setting
$\theta=\frac{\epsilon}{c+\epsilon-A}$.
\end{proof}


  \bibliographystyle{abbrv}


\end{document}